\documentclass[reqno]{amsart}

\usepackage{enumerate}
\usepackage{amsmath, amssymb, amsthm}
\usepackage{mathrsfs}
\usepackage{esint}
\usepackage[usenames,dvipsnames]{xcolor}\usepackage{mathtools}
\usepackage{hyperref}
\usepackage{bm}
\usepackage{kotex}
\usepackage{comment}
\hypersetup{%
	colorlinks=true, linkcolor=blue,
	citecolor=ForestGreen
}
\usepackage[paper=letterpaper,margin=1in]{geometry}
\usepackage{acronym}
\usepackage{dsfont}

\numberwithin{equation}{section}

\newtheorem{theorem}{Theorem}[section]
\newtheorem{lemma}[theorem]{Lemma}
\newtheorem{corollary}[theorem]{Corollary}
\newtheorem{remark}[theorem]{\bf{Remark}}
\newtheorem{assumption}[theorem]{Assumption}
\newtheorem{example}[theorem]{Example}
\newtheorem{definition}[theorem]{Definition}

\theoremstyle{remark}
\theoremstyle{definition}

\newcommand\bL{\mathbb{L}}
\newcommand\bR{\mathbb{R}}

\newcommand\bH{\mathbb{H}}

\newcommand\bZ{\mathbb{Z}}

\newcommand\bE{\mathbb{E}}
\newcommand\bN{\mathbb{N}}
\newcommand\bP{\mathbb{P}}

\newcommand\cB{\mathcal{B}}
\newcommand\cC{\mathcal{C}}

\newcommand\cF{\mathcal{F}}

\newcommand\cH{\mathcal{H}}

\newcommand\cL{\mathcal{L}}
\newcommand\cP{\mathcal{P}}

\newcommand\cS{\mathcal{S}}
\newcommand\cM{\mathcal{M}}

\newcommand\ep{\varepsilon}

\newcommand\cbrk{\text{$]$\kern-.15em$]$}}
\newcommand\opar{\text{\,\raise.2ex\hbox{${\scriptstyle
|}$}\kern-.34em$($}}
\newcommand\cpar{\text{$)$\kern-.34em\raise.2ex\hbox{${\scriptstyle |}$}}\,}

\makeatletter
\renewcommand{\tocsection}[3]{%
  \indentlabel{\@ifnotempty{#2}{\bfseries\ignorespaces#1 #2\quad}}\bfseries#3}
\renewcommand{\tocsubsection}[3]{%
  \indentlabel{\@ifnotempty{#2}{\ignorespaces#1 #2\quad}}#3}

\newcommand\@dotsep{4.5}
\def\@tocline#1#2#3#4#5#6#7{\relax
  \ifnum #1>\c@tocdepth 
  \else
    \par \addpenalty\@secpenalty\addvspace{#2}%
    \begingroup \hyphenpenalty\@M
    \@ifempty{#4}{%
      \@tempdima\csname r@tocindent\number#1\endcsname\relax
    }{%
      \@tempdima#4\relax
    }%
    \parindent\z@ \leftskip#3\relax \advance\leftskip\@tempdima\relax
    \rightskip\@pnumwidth plus1em \parfillskip-\@pnumwidth
    #5\leavevmode\hskip-\@tempdima{#6}\nobreak
    \leaders\hbox{$\m@th\mkern \@dotsep mu\hbox{.}\mkern \@dotsep mu$}\hfill
    \nobreak
    \hbox to\@pnumwidth{\@tocpagenum{\ifnum#1=1\bfseries\fi#7}}\par
    \nobreak
    \endgroup
  \fi}
\AtBeginDocument{
\expandafter\renewcommand\csname r@tocindent0\endcsname{0pt}
}
\def\l@subsection{\@tocline{2}{0pt}{2.5pc}{5pc}{}}
\makeatother

\begin{document}

\title[strong dissipativity]{\texorpdfstring{$L_p$}{Lg}-regularity theory for the stochastic reaction-diffusion equation with super-linear multiplicative noise and strong dissipativity
}

\author[B.-S. \ Han]{Beom-Seok Han}

{\color{red} \address{B.-S. \ Han, 
	{Department of Mathematical Sciences, Korea Advanced Institute of Science and Technology \newline\hphantom{\quad \ \ B.-S. \ Han} (KAIST),	291, Daehak-ro, Yuseong-gu, Daejeon, Republic of Korea}}

\email{bhan91@kaist.ac.kr}}

\author[J.\ Yi]{Jaeyun Yi }
\address{J.\ Yi,
	School of Mathematics, Korea Institute for Advanced Study (KIAS),
	\newline\hphantom{\quad \ \ J. Yi}
	85 Hoegiro Dongdaemun-gu, Seoul 02455, Republic of Korea
	}
\email{jaeyun@kias.re.kr}






\begin{abstract}
We study the existence, uniqueness, and regularity of the solution to the stochastic reaction-diffusion equation (SRDE) with colored noise $\dot{F}$:
$$
\partial_t u = a^{ij}u_{x^ix^j} + b^i u_{x^i} + cu - \bar{b} u^{1+\beta} + \xi u^{1+\gamma}\dot F,\quad (t,x)\in \bR_+\times\bR^d; \quad u(0,\cdot) = u_0,
$$ where $a^{ij},b^i,c, \bar{b}$ and $\xi$ are $C^2$ or $L_\infty$ bounded random coefficients. Here $\beta>0$ denotes the degree of the strong dissipativity and $\gamma>0$ represents the degree of stochastic force. Under the reinforced Dalang's condition on $\dot{F}$, we show the well-posedness of the SRDE provided $\gamma < \frac{\kappa(\beta +1)}{d+2}$ where $\kappa>0$ is the constant related to $\dot F$. Our result assures that strong dissipativity prevents the solution from blowing up. Moreover, we provide the maximal H\"older regularity of the solution in time and space.
\end{abstract}

\maketitle

%

\section{Introduction}

In this paper, we investigate the existence, uniqueness, $L_p$-regularity, and H\"older regularity of a solution to the following stochastic partial differential equation (SPDE):
\begin{equation}
\label{main_equation}
\partial_t u = a^{ij}u_{x^ix^j} + b^i u_{x^i} + cu - \bar{b} u^{1+\beta} + \xi u^{1+\gamma}\dot F,\quad (t,x)\in \bR_+\times\bR^d; \quad u(0,\cdot) = u_0.
\end{equation} where $\beta, \gamma>0$ and the coefficients $a^{ij},b^i,c, \bar{b}$ and $\xi$ are $\cP \times \cB(\bR^d)$-measurable real-valued maps with Einstein's summation convention on $i,j=1,...d$. Here $\dot{F}$ is a spatially homogeneous colored noise. More precisely, $\dot{F}$ is a centered generalized Gaussian random field whose covariance is given by 
\begin{equation*}
\bE[F(t,x)F(s,y)] = \delta_0(t-s)f(x-y),
\end{equation*}
where $\delta_0$ is the  Dirac delta distribution and $f$ is a {\it correlation function (or measure)}. In other words, $f$ is a nonnegative and nonnegative definite function/measure. We denote that the spectral measure  $\mu$ of $f$ is a nonnegative tempered measure defined by
\begin{equation}
 	\label{spectral measure}
 	\mu(\xi):=\cF(f)(\xi) := \frac{1}{(2\pi)^{d/2}}\int_{\bR^d} e^{-i\xi\cdot x } f(dx)
\end{equation} 
which is the Fourier transform of $f$.

The SPDE \eqref{main_equation} is also called the stochastic reaction-diffusion equation (SRDE) whose typical form is given by 
\begin{equation}\label{eq:SRDE}
	\partial_t u = \cL u + g(u) +\sigma(u)\dot{F},  \quad (t,x)\in \bR_+\times\bR^d; \quad u(0,\cdot) = u_0,
\end{equation} where $\cL$ is a second-order eilliptic differential operator. The function $g(u)$ denotes the reaction force and the term $\sigma(u)\dot{F}$ denotes the random force called the multiplicative noise. There have been successful investigations on the existence and uniqueness of the solution to the SRDE \eqref{eq:SRDE}. When $g$ and $\sigma$ are globally Lipschitz continuous with at most linear growth, the well-posedness of \eqref{eq:SRDE} was studied in \cite{da2014stochastic,dalang1999extending,walsh1986introduction,kry1999analytic}. For the case where $\sigma$ is locally Lipschitz continuous with linear growth and $g$ is non-Lipschitz, including the Fisher-KPP type ($g(u) = u-u^2$) and the Allen-Cahn type ($g(u)= u-u^3$), there have been also fruitful achievements such as \cite{cerrai2003stochastic,cerrai2011averaging,cerrai2013pathwise} under some monotonicity conditions on $g$ with polynomial growths (see also \cite{marinelli2010uniqueness,marinelli2018well}). We refer to \cite{mueller2011effect,khoshnevisan2020phase} for interesting topics on equations of this type.
 For $\sigma$ with super-linear growh, Mueller \cite{mueller1991long,mueller2000critical} showed that when $\cL = \Delta$ and $g=0$, the solution does blow-up at finite time with positive probability if $|\sigma(u)| \gtrsim|u|^{1+\gamma}$ for $\gamma>\frac{1}{2}$ and does not if $|\sigma(u)| \lesssim (1+|u|)^{1+\gamma}$ for $ 0 \leq\gamma< \frac{1}{2}$ (see also \cite{choi2021regularity} for colored noise cases). We refer to \cite{dalang2019global,chen2022superlinear} for the the case where $|g(u)| \lesssim |u|\log|u|$ and $|\sigma(u)| \lesssim |u|\log|u|^{\kappa/2}$ for some $\kappa>0 $ satisfying the reinforced Dalang's condition on $f$ (see Assumption \ref{assumption on f}).

In the present paper, we discuss the well-posedness problem of \eqref{eq:SRDE} when $g$ is given by the {\it strong dissipative term} $g(u) = -\bar{b}u^{1+\beta}$  and {\it super-linear multiplicative term} $\sigma(u)= \xi u^{1+\gamma}$ for $\beta,\gamma>0$ satisfying $\gamma < {\frac{\kappa(1+\beta)}{d+2}}$. We also assume that the $\cL$ is a second order differential operator with random coefficients. Our purpose is to show that the strong dissipativity of $g$ (denoted by $\beta$) push down the solution to the finite value, thus it allows the random force of $\sigma$ (denoted by $\gamma$) to be stronger. In other words, given any large $\gamma>0$, we prove that a unique solution exists in a suitable Sobolev space (possibly embedded into a H\"older space) until an arbitrary bounded stopping time provided that $\beta>0$ is sufficently large.

In a recent work \cite{salins2022global}, Salins has recently addressed a similar problem on an open bounded domain. He proved that when $g(u) \lesssim -u|u|^{1+\beta}$ and $|\sigma(u)| \lesssim (1+ |u|)^{1+\gamma}$, a local mild solution is indeed a global solution (see Theorem 2.7 of \cite{salins2022global}) if $\gamma<\frac{(1-\eta)\beta}{2}$, where $\eta<1$ is given in (2.9) of \cite{salins2022global}. His proof idea hinges on the estimate of stochasic convolutions with nicely combinating a stopping time argument, which is different from ours (see Remark \ref{rmk: relation between beta and gamma}). In Remark \ref{comparison with salins' result}, we discuss the comparison of our result with Salins'. 

We mention a few words on our contribution and novelty. First, employing the $L_p$-theory developed by Krylov \cite{kry1999analytic}, we obtain the existence, uniqueness, and regularity of the solution to \eqref{main_equation} in $L_p$ spaces. Additionally, our approach allows for random coefficients $a^{ij},b^i,c, \bar{b}$ and $\xi$ (see Assumption \ref{assumption:coefficients}). Moreover, our setting considers $\bR^d$ as the spatial domain, instead of an open bounded domain $D$ which was studied in \cite{salins2022global}. Since our proof does not rely on the series representation of the noise using the eigenfunctions of $\cL$ in $L^2(D)$ with the Dirichlet boundary condition, our results are achieved on the whole spatial domain $\bR^d$. To the best of our knowledge, our main theorem (Theorem \ref{thm:main}) is the first result that establishes the well-posedness of \eqref{main_equation} on $\bR^d$ with $\gamma>\frac{1}{2}$, which is essentially due to the strong dissipativity. Furthermore, we obtain the H\"older regularity of the solution in time and space (see Corollary \ref{Holder regularity of the solution}):
\begin{equation}
\label{eq:Holder regularity of the solution}
u\in C_{t,x}^{\frac{1}{2}\left[ \kappa-\frac{\gamma}{1+\beta}(d+2) \right]-\ep,\kappa-\frac{\gamma}{1+\beta}(d+2)-\ep}([0,T]\times\bR^d)
\end{equation}
almost surely, where $\kappa$ is the constant related to reinforced Dalang's condition (see Assumption \ref{assumption on f}).  Our result is new in that \eqref{eq:Holder regularity of the solution} shows the regularizing effect of the strong dissipativity or a large $\beta$.

Now we pose to give proof ideas of Theorem \ref{thm:main}. Our first step is to construct a local solution $u_m$ (see Lemma \ref{cut_off_lemma}) using the $L_p$-theory of SPDEs (Theorem \ref{theorem_nonlinear_case}) whose reaction term is bounded uniformly (see Lemma \ref{Lq bounds}). We then introduce a auxillary process $v_m$, a solution to a SPDE related to the nonlinear noise part of $u_m$ without the reaction term, which dominates $u_m$, i.e., $u_m\leq v_m$ (see Lemmas \ref{lem:well-posedness of v_m} and \ref{lem:upper bound of u_m by v_m}). We show that $v_m$ (hence $u_m$) is controlled by the reaction term due to the strong dissipativity condition $\gamma<\frac{\kappa(1+\beta)}{d+2}$ (see the choice of $p$ and $p_0$ in the proof of Lemma \ref{lem:noise related solution control}). We finally obtain a uniform bound on $u_m$ in Lemma \ref{lem:uniform bound of local solutions in probability}, which is the essential part of the proof of Theorem \ref{thm:main} (see Section \ref{subsec:proof of main thm}).

Our paper is organized as follows: In Section \ref{sec:preliminaries}, we collect some preliminiaries for SPDEs with colored noise. Section \ref{sec:main_results} presents the main result (Theorem \ref{thm:main}). Section \ref{sec:proof of main thm} is devoted to proving Theorem \ref{thm:main}.
 


We finish the introduction with some notations and conventions. Let $\bN$ and $\bR$ denote the set of natural numbers and real numbers, respectively. We write $:=$ to denote definition.  For a real-valued function $g$, we set
\begin{equation*}
g_+:= \frac{g+|g|}{2},\quad g_- = -\frac{g-|g|}{2}
\end{equation*} For a normed space $F$, a measure space $(X,\mathcal{M},\mu)$, and $p\in [1,\infty)$, a space $L_{p}(X,\cM,\mu;F)$ is a set of $F$-valued $\mathcal{M}^{\mu}$-measurable function such that
\begin{equation}
\label{def_norm}
\| u \|_{L_{p}(X,\cM,\mu;F)} := \left( \int_{X} \| u(x) \|_{F}^{p}\mu(dx)\right)^{1/p}<\infty.
\end{equation}
A set $\mathcal{M}^{\mu}$ is the completion of $\cM$ with respect to the measure $\mu$.
For $\alpha\in(0,1]$ and $T>0$, a set $C^{\alpha}([0,T];F)$ is the set of $F$-valued continuous functions $u$ such that
$$ |u|_{C^{\alpha}([0,T];F)}:=\sup_{t\in[0,T]}|u(t)|_{F}+\sup_{\substack{s,t\in[0,T], \\ s\neq t}}\frac{|u(t)-u(s)|_F}{|t-s|^{\alpha}}<\infty.$$
For $a,b\in \bR$, set
$a \wedge b := \min\{a,b\}$, $a \vee b := \max\{a,b\}$.
Let $\cS = \cS(\bR^d)$ denote the set of Schwartz functions on $\bR^d$. The Einstein's summation convention with respect to $i,j$, and $k$ is assumed. A generic constant $N$ can be different from line to line. We denote by $N=N(a,b,\ldots)$ or $N=N_{a,b,\ldots}$ that $N$ depends only on $a,b,\ldots$. For functions depending on $\omega$, $t$, and $x$, the argument $\omega \in \Omega$ is usually omitted.

\section{Preliminaries} 
\label{sec:preliminaries}

This section is devoted to introducing preliminaries related to \eqref{main_equation}. We first review the definition of the spatially homogeneous colored noise $\dot F$ and the reinforced Dalang's condition on $\dot{F}$ which implies the existence, { uniqueness} and the H\"older continuity of the solution to \eqref{main_equation} (see \cite{sanz2002holder,ferrante2006spdes}). We then introduced the definitions and properties of stochastic Banach spaces $\bH_p^{n+2}(\tau)$ and $\cH_p^{n+2}(\tau)$, where in particular the latter is our solution space (see \cite{kry1999analytic}).

Let $(\Omega, \cF, P)$ be a complete probability space equipped with a filtration $\{\cF_t\}_{t\geq0}$. The filtration $\{\cF_t\}_{t\geq0}$ is assumed to satisfy the usual conditions. A set $\cP$ denotes the predictable $\sigma$-field related to $\{\cF_t\}_{t\geq0}$. First, we provide definitions and assumptions for the noise $\dot F$.

\begin{definition}
\label{def of F}
$\{F(\phi):\phi\in\cS(\bR^{d+1})\}$ on $(\Omega, \cF, P)$ is a mean-zero Gaussian process with covariance functional given by 
\begin{equation}
\label{def of covaiance functional}
\begin{aligned}
\bE \left[F(\phi)F(\psi)\right] 
&= \int_0^\infty  dt \int_{\bR^d}  f(dx)\, (\phi(t,\cdot) *  \tilde \psi(t, \cdot))(x)  \\
& = \int_0^\infty dt  \int_{\bR^d} \mu(d\xi)\,  \hat \phi(t, \xi) \, \overline{\hat \psi(t, \xi)}, 
\end{aligned}
\end{equation}
where $\tilde \psi (t, x):=\psi(t, -x)$, $\hat\phi(t, \xi):= \cF(\phi(t,\cdot)) (\xi)$ is the Fourier transform of $\phi$ with respect to $x$,  $f$ is a nonnegative, and nonnegative definite tempered measure, and $\mu$ is the Fourier transform of $f$ introduced in \eqref{spectral measure}. We say that $f$ is a correlation function or measure.
\end{definition}

\begin{remark} \label{remark:representation}

By the framework of Dalang-Walsh theory (see \cite{walsh1986introduction,dalang1999extending}), the $\cS(\bR^{d+1})$-indexed Gaussian process $F$ can be extended to an $L^2(\Omega)$-valued martingale measure $F(ds,dx)$. Additionally, it should be remarked that the integral with $L^2(\Omega)$-valued martingale measure $F(ds,dx)$ can also be written as an infinite summation of It\^o stochastic integral, which we use throughout the paper. For example, if a predictable process $X(t,\cdot)$ is given as
$$ X(t,x) = \zeta(x)1_{ \opar\tau_1,\tau_2\cbrk}(t),
$$
where $\tau_1$, $\tau_2$ are bounded stopping times, $\opar\tau_1,\tau_2\cbrk:=\{ (\omega,t):\tau_1(\omega)<t\leq \tau_2(\omega) \}$, and $\zeta\in C_c^\infty$, then we have 
\begin{equation*} 
\int_{0}^t\int_{\bR^d}X(s,x)F(ds,dx)=\sum_{k=1}^{\infty}\int_0^t \int_{\bR^d}(f\ast e_k)(x) X(s,x) dxdw^k_s,
\end{equation*}
where $\{w_t^k, k \in \bN\}$ is a collection of  one-dimensional independent Wiener processes and $\{ e_k, k \in \bN\}\subseteq \cS$ is a complete orthonormal basis of a Hilbert space $\cH$ induced by $f$; see \cite{ferrante2006spdes,choi2021regularity}. The construction and properties of $\cH$ and $\{e_k, k\in\bN \}$ are described in \cite[Remark 2.6]{choi2021regularity}.
\end{remark}

We assume the reinforced Dalang's condition for $f$. This assumption is required to obtain the solvability of the stochastic partial differential equation driven by spatially homogeneous colored noise (e.g., see \cite[Remark 4]{ferrante2006spdes}).

\begin{assumption}\label{assumption on f}
The correlation function $f$ satisfies  the following: For some $\kappa\in(0,1]$,
\begin{equation}
\label{reinforced Dalang's condition}
\begin{aligned}
\begin{cases}
\int_{|x|<1} |x|^{2-2\kappa -d} \, f(dx)  < +\infty \quad &\text{if}\quad 0<1-\kappa<\frac{d}{2},	\\
\int_{|x|<1} \log\left(\frac{1}{|x|}\right)\, f(dx) < +\infty \quad &\text{if}\quad 1-\kappa=d/2, \\
\text{no conditions on }f &\text{if}\quad  1-\kappa>d/2.
\end{cases}
\end{aligned}
\end{equation} 

\end{assumption}

\begin{example}[Examples of noise] Assumption \ref{assumption on f} covers a large class of correlation functions. For example, $f(dx) = \delta_0(dx)$ (space-time white noise) satisfies \eqref{reinforced Dalang's condition} for $d=1$ and $\kappa\in(0,1/2)$. One may consider the Riesz type kernel $f(dx) = |x|^{-\alpha}$ with $\alpha\in(0,2(1-\kappa)\wedge d)$ for $\kappa\in(0,1)$. Moreover, the Ornstein-Uhlenbeck type kernel $f(x) = \exp(-|x|^\beta)$ with $\beta \in (0,2]$ and the Brownian motion $f(x)\equiv 1$ also satisfy Assumption \ref{assumption on f} with any $\kappa\in(0,1)$.
\end{example}


\begin{remark}
  The condition on $f$ given in Assumption \ref{assumption on f} is equivalent to 
  \begin{equation}\label{reinforced Dalang's condition_equivalent form}
  \int_{\bR^d} \frac{\mu(d \xi)}{(1+|\xi|^2)^{1-\kappa}} < +\infty \quad \text{for some $\kappa\in(0,1]$,}
\end{equation} where $\mu$ is the spectral measure of $f$ defined in \eqref{spectral measure} (see \cite[Proposition 5.3]{sanz2000path}). The reinforced Dalang's condition guarantees the existence, uniqueness, and the H\"older continuity of the solution when $g$ and $\sigma$ in \eqref{eq:SRDE} are globally Lipschitz (see \cite[Theorem 6]{ferrante2006spdes} and \cite[Theorem 3.5]{choi2021regularity}).
\end{remark}

Next, we introduce definitions and properties of the Bessel potential spaces and stochastic Banach spaces. For more details, we refer the reader to \cite{grafakos2009modern,kry1999analytic,krylov2008lectures}.

\begin{definition}[Bessel potential space]
Let $p>1$ and $n \in \bR$. The space $H_p^n=H_p^n(\bR^d)$ is the set of all tempered distributions $u$ on $\bR^d$ satisfying

$$ \| u \|_{H_p^n} := \left\| (1-\Delta)^{n/2} u\right\|_{L_p} = \left\| \cF^{-1}\left[ \left(1+|\xi|^2\right)^{n/2}\cF(u)(\xi)\right]\right\|_{L_p}<\infty.
$$
Similarly, $H_p^n(\ell_2) = H_p^n(\bR^d;\ell_2)$ is a space of $\ell_2$-valued functions $g=(g^1,g^2,\cdots)$ satisfying
$$ \|g\|_{H_{p}^n(\ell_2)}:= \left\| \left| \left(1-\Delta\right)^{n/2} g\right|_{l_2}\right\|_{L_p} = \left\| \left|\cF^{-1}\left[ \left(1+|\xi|^2\right)^{n/2}\cF(g)(\xi)\right]\right|_{\ell_2} \right\|_{L_p}
< \infty. 
$$
For $n = 0$, we set $L_p := H_p^0$ and $L_p(\ell_2) := H_p^0(\ell_2)$.
\end{definition}

\begin{remark} \label{Kernel}
It is well-known that for $n\in (0,\infty)$ and $u\in \cS$, we have
\begin{equation*}
(1-\Delta)^{-n/2}u(x)=\int_{\bR^d}R_{n}(x-y)u(y)dy,
\end{equation*}
where 
\begin{equation*} 
|R_n(x)| \leq N(n,d)\left(e^{-|x|/2}1_{|x|\geq2} + A_n(x)1_{|x|<2}\right)
\end{equation*}
and
\begin{equation*}
\begin{aligned}
A_{n}(x):=
\begin{cases}
|x|^{n-d} + 1 + O(|x|^{n-d+2}) \quad &\mbox{if} \quad 0<n<d,\\ 
\log(2/|x|) + 1 + O(|x|^{2}) \quad &\mbox{if} \quad n=d,\\ 
1 + O(|x|^{n-d}) \quad &\mbox{if} \quad n>d.
\end{cases}
\end{aligned}
\end{equation*}
Additionally, notice that
$$ (R_{n} \ast R_{n})(x) = R_{2n}(x).
$$
For more information, see \cite[Proposition 1.2.5.]{grafakos2009modern}.

\end{remark}

\begin{remark}
Observe that Assumption \ref{assumption on f} yields
\begin{equation}
\label{def of nu kappa}
\nu_{\kappa}:=\int_{\bR}R_{2(1-\kappa)}(x)f(dx) < \infty,
\end{equation}
where $\kappa$ is the constant introduced in Assumption \ref{assumption on f}. The finiteness of $\nu_k$ is essential to obtain the unique solvability of local solutions of \eqref{main_equation}. See Lemma \ref{cut_off_lemma} and \eqref{cut_off_lower_order_cal_2}.

\end{remark}

Below, we introduce the space of pointwise multipliers for $H_p^n$.

\begin{definition}
\label{def_pointwise_multiplier}
Fix $n\in\bR$ and $\alpha\in[0,1)$ such that $\alpha = 0$ if $n\in\bZ$ and $\alpha>0$ if $|n|+\alpha$ is not an integer. Define
\begin{equation*}
\begin{aligned}
B^{|n|+\alpha} = 
\begin{cases}
B(\bR^d) &\quad\text{if } n = 0, \\
C^{|n|-1,1}(\bR^d) &\quad\text{if $n$ is a nonzero integer}, \\
C^{|n|+\alpha}(\bR^d) &\quad\text{otherwise},
\end{cases}
\end{aligned}
\end{equation*}
where $B(\bR^d)$ is the space of bounded Borel functions on $\bR^d$, $C^{|n|-1,1}(\bR^d)$ is the space of $|n|-1$ times continuous differentiable functions whose derivatives of $(|n|-1)$-th order derivative are Lipschitz continuous, and $C^{|n|+\alpha}$ is the real-valued H\"older spaces. The space $B(\ell_2)$ denotes a function space with $\ell_2$-valued functions, instead of real-valued function spaces.

\end{definition}

Next, we review the properties of Bessel potential space $H_p^n$.

\begin{lemma}
\label{prop_of_bessel_space} 

Let $p>1$ and $n \in \bR$. 
\begin{enumerate}[(i)]
\item 
\label{dense_subset_bessel_potential}
The space  $C_c^\infty(\bR^d)$ is dense in $H_{p}^{n}$. 

\item
\label{sobolev-embedding} 
Let $n - d/p = j+\alpha$ for some $j=0,1,\cdots$ and $\alpha\in(0,1]$. Then, for any  $i\in\{ 0,1,\cdots,j \}$, we have
\begin{equation} 
\label{holder embedding}
\left| D^i u \right|_{C(\bR^d)} + \left[D^n u\right]_{\cC^\alpha(\bR^d)} \leq N \| u \|_{H_{p}^n},
\end{equation}
where $N = N(n,p)$ and $\cC^\alpha$ is a Zygmund space.

\item
\label{bounded_operator}
The operator $D_i:H_p^{n}\to H_p^{n+1}$ is bounded. Moreover, for any $u\in H_p^{n+1}$,
$$ \left\| D^i u \right\|_{H_p^n} \leq N\| u \|_{H_p^{n+1}},
$$
where $N = N(n,p)$.

\item
\label{norm_bounded}
Let $n'\leq n$ and $u\in H_p^n$. Then, $u\in H_p^{n'}$ and 
$$ \| u \|_{H_p^{n'}} \leq \| u \|_{H_p^n}. 
$$

\item 
\label{iso} (isometry). For any $n',n\in\bR$, the operator $(1-\Delta)^{n'/2}:H_p^n\to H_p^{n-n'}$ is an isometry.

\item
\label{multi_ineq} (multiplicative inequality). Let 
\begin{equation*} \label{condition_of_constants_interpolation}
\begin{gathered}
\ep\in[0,1],\quad p_i\in(1,\infty),\quad n_i\in \bR,\quad i=0,1,\\
n=\ep n_0+(1-\ep)n_1,\quad1/p=\ep/p_0+(1-\ep)/p_1.
\end{gathered}
\end{equation*}
Then, we have
\begin{equation*}
\|u\|_{H^n_{p}} \leq \|u\|^{\ep}_{H^{n_0}_{p_0}}\|u\|^{1-\ep}_{H^{n_1}_{p_1}}.
\end{equation*}

\item \label{pointwise_multiplier}
Let $u\in H_p^n$. Then, we have
\begin{equation*}
\| au \|_{H_p^n} \leq N(n,p)\| a \|_{B^{|n|+\alpha}}\| u \|_{H_p^n}\quad\text{and}\quad\| bu \|_{H_p^n(\ell_2)} \leq N(n,p)\| b \|_{B^{|n|+\alpha}(\ell_2)}\| u \|_{H_p^n},
\end{equation*}
where $B^{|n|+\alpha}$ and $B^{|n|+\alpha}(\ell_2)$ are introduced in Definition \ref{def_pointwise_multiplier}.
\end{enumerate}
\end{lemma}
\begin{proof}
For \eqref{dense_subset_bessel_potential}-\eqref{multi_ineq}, see Theorem 13.3.7 (i), Theorem 13.8.1, Theorem 13.3.10, Corollary 13.3.9, Theorem 13.3.7 (ii), Exercise 13.3.20 of \cite{krylov2008lectures}, respectively. For \eqref{pointwise_multiplier}, see  \cite[Lemma 5.2]{kry1999analytic}.
\end{proof}

Throughout the paper, let $T>0$ be a nonrandom constant. We introduce stochastic Banach spaces $\bH_p^{n}(\tau)$ and $\cH_p^{n}(\tau)$. For more details, see \cite[Section 3]{kry1999analytic}.

\begin{definition}[Stochastic Banach spaces]
For a bounded stopping time $\tau\leq T$, let us denote $\opar0,\tau\cbrk:=\{ (\omega,t):0<t\leq \tau(\omega) \}$.
\begin{enumerate}[(i)]
\item \label{def:Hp,Up}
For $\tau\leq T$, 
\begin{gather*}
\bH_{p}^{n}(\tau) := L_p(\opar0,\tau\cbrk, \mathcal{P}, d P \times dt ; H_{p}^n),\\
\bH_{p}^{n}(\tau,\ell_2) := L_p(\opar0,\tau\cbrk,\mathcal{P}, dP \times dt;H_{p}^n(\ell_2)),\\
U_{p}^{n} :=  L_p(\Omega,\cF_0, dP ; H_{p}^{n-2/p}).
\end{gather*}	
For convenience, we write $\bL_p(\tau):=\bH^{0}_{p}(\tau)$ and $\bL_p(\tau,\ell_2):=\bH^{0}_{p}(\tau,\ell_2)$.

\item \label{def:Hp-gamma}
The norm of each space is defined in the natural way. For example, 
\begin{equation} \label{norm}
\| u \|^p_{\bH^{n}_{p}(\tau)} := \bE \left[\int^{\tau}_0 \| u(t) \|^p_{H^{n}_{p}}dt\right]. 
\end{equation}

\end{enumerate}
\end{definition}

\begin{definition}[Solution spaces] 
\label{definition_of_sol_space}
Let $\tau\leq T$ be a bounded stopping time and $p\geq2$. 

\begin{enumerate}[(i)]
\item 
For $u\in \bH_p^{n+2}(\tau)$, we write $u\in\cH^{n+2}_p(\tau)$ if there exists $u_0\in U_{p}^{n+2}$ and  $(h,g)\in
\bH_{p}^{n}(\tau)\times\bH_{p}^{n+1}(\tau,\ell_2)$ such that
\begin{equation*}
du = hdt+\sum_{k=1}^{\infty} g^k dw_t^k,\quad   t\in (0, \tau]\,; \quad u(0,\cdot) = u_0
\end{equation*}
in the sense of distributions. In other words, for any $\phi\in \cS$, the equality
\begin{equation} \label{def_of_sol}
(u(t,\cdot),\phi) = (u_0,\phi) + \int_0^t(h(s,\cdot),\phi)ds + \sum_{k=1}^{\infty} \int_0^t(g^k(s,\cdot),\phi)dw_s^k
\end{equation}
holds for all $t\in [0,\tau]$ almost surely. Here, $\left\{w_t^k:k\in \bN\right\}$ is a set of one-dimensional independent Wiener process. 

\item

The norm is defined as
\begin{equation}
\label{def_of_sol_norm}
\| u \|_{\cH_{p}^{n+2}(\tau)} :=  \| u \|_{\mathbb{H}_{p}^{n+2}(\tau)} + \| h \|_{\mathbb{H}_{p}^{n}(\tau)} + \| g \|_{\mathbb{H}_{p}^{n+1}(\tau,\ell_2)} + \| u_0 \|_{U_{p}^{n+2}}.
\end{equation}

\item \label{def_of_local_sol_space}
For a stopping time $\tau \in [0,\infty]$, we write $u \in \cH_{p,loc}^{n+2}(\tau)$ if there exists a sequence of bounded stopping times $\{ \tau_n : n\in\bN \}$ such that $\tau_n\uparrow \tau$ (a.s.) as $n\to\infty$ and $u\in \cH_{p}^{n+2}(\tau_n)$ for each $n$. We write $u = v$ in $\cH_{p,loc}^{n+2}(\tau)$ if there exists a sequence of bounded stopping times $\{ \tau_n : n\in\bN \}$ such that $\tau_n\uparrow\tau$ (a.s.) as $n\to\infty$ and $u = v$ in $\cH_{p}^{n+2}(\tau_n)$ for each $n$. We omit $\tau$ if $\tau = \infty$. In other words,  $\cH_{p,loc}^{n+2}=\cH_{p,loc}^{n+2}(\infty)$.

\item If $n+2 = 0$, we use $\cL$ instead of $\cH$. For example, $\cL_p(\tau) := \cH_p^0(\tau)$.

\end{enumerate}

\end{definition}





In what follows, we present properties of the solution space $\cH_p^{n+2}(\tau)$ and the H\"older embedding theory.

\begin{theorem} 
\label{embedding}
Let $\tau\leq T$ be a bounded stopping time.
\begin{enumerate}[(i)]

\item \label{completeness}
For any $p\geq2$, $n\in\bR$, $\cH_p^{n+2}(\tau)$ is a Banach space with the norm $\| \cdot \|_{\cH_p^{n+2}(\tau)}$.

\item \label{large-p-embedding} 
If $p>2$, $n\in\bR$, and $1/p < \alpha_1 < \alpha_2 < 1/2$, then for any  $u\in\cH_{p}^{n+2}(\tau)$, we have $u\in C^{\alpha_{1}-1/p}\left([0,\tau];H_{p}^{n+2-2\alpha_2}\right)$ (a.s.) and
\begin{equation} 
\label{solution_embedding}
\mathbb{E}| u |^p_{C^{\alpha_1-1/p}\left([0,\tau];H_{p}^{n + 2 - 2\alpha_2} \right)} \leq N(\alpha_1,\alpha_2,n,d,p,T)\| u \|^p_{\cH_{p}^{n+2}(\tau)}.
\end{equation}

\item \label{gronwall_type_ineq} 
Let $p > 2$, $n\in\bR$, and $u\in\cH_p^{n+2}(\tau)$. If there exists $n' < n$ such that
\begin{equation*}
\| u \|_{\cH_p^{n+2}(\tau\wedge t)}^p \leq N_0 + N_1 \| u \|_{\bH_p^{n'+2}(\tau\wedge t)}^p
\end{equation*}
for all $t\in(0,T)$, then we have
\begin{equation} \label{modified_Gronwall}
\| u \|_{\cH_p^{n+2}(\tau\wedge T)}^p \leq N_0N,
\end{equation}
where $N =  N(n,d,p,N_1,T)$.

\end{enumerate}
\end{theorem}
{\begin{proof}

For \eqref{completeness} and \eqref{large-p-embedding}, see Theorem 3.7 and Theorem 7.2 of \cite{kry1999analytic}, respectively. Thus, we only prove \eqref{gronwall_type_ineq}. Because $\gamma_0 < \gamma$, Lemma \ref{prop_of_bessel_space} \eqref{multi_ineq} yields
\begin{equation*}
\begin{aligned}
\|u\|_{\cH_p^{\gamma+2}(\tau\wedge t)}^p &\leq N_0 +  N_1 \| u  \|_{\bH^{\gamma_0+2}_p(\tau\wedge t)}^p \\
&\leq N_0 +  \frac{1}{2} \bE\int_0^{\tau\wedge t} \| u(s,\cdot) \|_{H^{\gamma+2}_p}^p ds + N \bE\int_0^{\tau\wedge t} \| u(s,\cdot) \|_{H^{\gamma}_p}^p ds\\
&\leq N_0 + \frac{1}{2}  \| u \|_{\cH^{\gamma+2}_p(\tau\wedge t)}^p + N \int_0^{t} \bE \sup_{r\leq \tau\wedge s}\| u(r,\cdot) \|_{H^{\gamma}_p}^p ds, \\
\end{aligned}
\end{equation*}
where $N = N(d,p,\gamma)$.
By subtraction and \eqref{solution_embedding}, we have
\begin{equation*}
\begin{aligned}
\|u\|_{\cH_p^{\gamma+2}(\tau\wedge t)}^p
&\leq 2N_0 + 2N \int_0^{t} \| u \|_{\cH^{\gamma+2}_p(\tau\wedge  s)}^p ds,
\end{aligned}
\end{equation*}
where $N = N(d,p,\gamma,T)$. By Gr\"onwall's inequality, we have \eqref{modified_Gronwall}. The theorem is proved.
\end{proof}}

\begin{corollary} 
\label{embedding_corollary}

Let $\tau\leq T$ be a bounded stopping time.
Suppose $n\in (0,1)$, $p\in[2,\infty)$, $\alpha_1,\alpha_2\in(0,\infty)$ satisfy
\begin{equation} \label{condition_for_alpha_beta}
\frac{1}{p} < \alpha_1 < \alpha_2 <  \frac{1}{2}\left(n-\frac{d}{p}\right).
\end{equation}
Then, we have $u \in C^{\alpha_1-1/p}([0,\tau];C^{n-2\alpha_2-1/p}(\bR^d) )$ almost surely and 
\begin{equation} \label{sol_embedding}
\bE \|u\|^p_{C^{\alpha_1-1/p}([0,\tau];C^{n-2\alpha_2-d/p}(\bR^d) )}\leq N(\alpha_1,\alpha_2,\kappa,d,n,p,T)\|u\|_{\cH^{n}_{p} (\tau)}^p.
\end{equation}

\end{corollary}
\begin{proof}

In Lemma \ref{prop_of_bessel_space} \eqref{sobolev-embedding}, consider $n-2\alpha_2$ instead of $n$. Then, we have
\begin{equation} 
\label{holder_lp_embedding}
\begin{aligned}
\|u(t,\cdot)\|_{C^{n-2\alpha_2-d/p}}&\leq N\| u(t,\cdot) \|_{H_{p}^{n-2\alpha_2}}
\end{aligned}
\end{equation}
for all $t \in [0,\tau]$ almost surely. By \eqref{holder_lp_embedding} and \eqref{solution_embedding}, we have \eqref{sol_embedding}. The corollary is proved.

\end{proof}

{We end this section with the definition of the solution to \eqref{main_equation}.}

{\begin{definition} Let $\tau\leq T$ be a bounded stopping time, $n\in\bR$, $ p \geq 2$, and Assumption \ref{assumption on f} hold. Then, we say that $u\in\cH_{p,loc}^n(\tau)$ is a solution to \eqref{main_equation} if for any $\phi \in \cS$, the equality
\begin{equation*}
\begin{aligned}
(u(t,\cdot),\phi) 
&= (u_0,\phi) + \int_0^t\left( a^{ij}(s,\cdot)u_{x^ix^j}(s,\cdot)+b^i(s,\cdot)u_{x^i}(s,\cdot)+c(s,\cdot)u(s,\cdot),\phi \right)ds \\
&\quad\quad -\int_0^t(\bar b^i(s,\cdot)u^{1+\beta}(s,\cdot),\phi)ds + \sum_{k=1}^{\infty}\int_0^t\int_{\bR^d} \xi(s,x)u^{1+\gamma}(s,x)(f\ast e_k)(x) \phi(x)dxdw^k_s
\end{aligned}
\end{equation*}
holds for all $t\in[0,\tau]$ almost surely, where $\{e_k,k\in \bN\}$ is a set functions introduced in Remark \ref{remark:representation}.

\end{definition}}

\vspace{2mm}

\section{Main results} 
\label{sec:main_results}

In this section, we present the assumptions on coefficients and the main results of this paper. Assumption \ref{assumption:coefficients} includes conditions on the measurability, ellipticity, and boundedness of the coefficients of \eqref{main_equation}. Theorem \ref{thm:main} establishes the existence, uniqueness, and $L_p$ regularity of the solution to \eqref{main_equation}. Furthermore, the H\"older regularity of the solution is obtained using the H\"older embedding theorem for $\cH_p^\kappa(\tau)$. Theorem \ref{uniqueness of solution in p} presents the inclusion relation of the solution spaces $\cH_p^\kappa(\tau)\subset\cH_q^\kappa(\tau)$ when $p<q$ and initial data satisfy additional conditions on summability. Finally, we combine Theorems \ref{thm:main} and \ref{uniqueness of solution in p} to derive the maximal H\"older regularity of the solution, which is given in Corollary \ref{Holder regularity of the solution}.

\begin{assumption} 
\label{assumption:coefficients}
\begin{enumerate}[(i)]

\item 
The coefficients $a = a(t,x)$, $b = b(t,x)$, $c = c(t,x)$, $\bar b = \bar b(t,x)$, and $\xi = \xi(t,x)$ are $\cP\times\cB(\bR^d)$-measurable.

\item 
There exists $K>0$ such that 
\begin{equation}
\label{ellipticity_of_leading_coefficients} 
K^{-1}|\eta|^2 \leq a^{ij}(t,x)\eta^i\eta^j \leq  K|\eta|^2 \quad \forall (\omega,t,x)\in\Omega\times[0,\infty)\times\bR^d\quad\text{and}\quad \eta = (\eta^1,\dots,\eta^d)\in\bR^d, 
\end{equation}
\begin{equation}
\label{lower bound of bar b} 
\bar b(t,x) \geq K^{-1}\quad \forall (\omega,t,x)\in\Omega\times[0,\infty)\times\bR^d, 
\end{equation}
and
\begin{equation} 
\label{boundedness_of_deterministic_coefficients} 
\sum_{i,j}\left| a^{ij}(t,\cdot) \right|_{C^{2}(\bR^d)} + \sum_{i}\left| b^i(t,\cdot) \right|_{C^{2}(\bR^d)} + |c(t,\cdot)|_{C^{2}(\bR^d)} + |\bar b(t,\cdot)|_{C^{2}(\bR^d)} + |\xi(t,\cdot)|_{L_\infty(\bR^d)} \leq K
\end{equation}
for all $(\omega,t)\in\Omega\times[0,\infty)$.
\end{enumerate}
\end{assumption}

\begin{remark}

In order to establish the global existence of a solution to \eqref{main_equation}, we need to have a uniform bound of local solutions. Therefore, we first obtain a uniform $L_p$ bound of local solutions. The lower bound of $\bar b$ is crucial in obtaining this uniform bound (see Lemma \ref{Lq bounds} and \eqref{L1_bound_calculation_2}). Intuitively, the lower bound of $\bar b$ prevents the dissipative term $u^{1+\beta}$ from vanishing, which in turn allows for the long time existence of a solution.

\end{remark}

Below we provide the main result of our paper.

\begin{theorem}
\label{thm:main}
Suppose Assumptions \ref{assumption on f} and \ref{assumption:coefficients} hold. Let $\beta,\gamma>0$ and $p\geq2$ satisfy
\begin{equation}
\label{conditions on beta and gamma}
\gamma <  \frac{\kappa(1+\beta)}{d+2}\quad\text{and}\quad p \in \left( \frac{d+2}{\kappa}, \frac{1+\beta}{\gamma} \right),
\end{equation}
where $\kappa$ is the constant introduced in Assumption \ref{assumption on f}. Then, for {nonnegative} initial data $u_0\in L_1(\Omega,\cF_0;L_1(\bR^d))\cap U_p^{\kappa}$, \eqref{main_equation} has a unique nonnegative solution $u\in\cH_{p,loc}^{\kappa}$. Furthermore, for any bounded stopping time $\tau\leq T$ and $\alpha_1$ and $\alpha_2$ satisfying 
\begin{equation}
\label{condition for alpha1 and alpha 2}
\frac{1}{p} < \alpha_1 < \alpha_2 <  \frac{1}{2}\left(\kappa-\frac{d}{p}\right),
\end{equation}
we have $u \in C^{\alpha_1-1/p}([0,\tau];C^{\kappa-2\alpha_2-d/p}(\bR^d) )$ almost surely and 
\begin{equation}
\label{holder regularity of solution}
\bE |u|^p_{C^{\alpha_1-1/p}([0,\tau];C^{\kappa-2\alpha_2-d/p}(\bR^d) )}\leq N(\alpha_1,\alpha_2,\kappa,d,p,T)\|u\|_{\cH^{\kappa}_{p} (\tau)}^p.
\end{equation}

\end{theorem}

\begin{remark}[Comparison with Salins' result \cite{salins2022global}]\label{comparison with salins' result}
\cite[Theorem 2.7]{salins2022global} identifies a condition that the solutions of \eqref{eq:SRDE} with $g(u)\text{sign}(u) \lesssim |u|^{1+\beta}$ and $\sigma(u) \lesssim 1+ |u|^{1+\gamma}$ on a bounded domain $D$ never explode when $\gamma<\frac{(1-\eta)\beta}{2}$ where $\eta<1$ is some constant related to the correlation function $f$ (see \cite[Assumption 2.2]{salins2022global}). Since $\eta$ is associated with the eigenfunction of $\cL$ on $L^2(D)$, it is difficult to compare the range of ($\beta,\gamma$) with ours ($\gamma<\frac{\kappa(1+\beta)}{d+2}$). The only case that can be identified is the ($1+1$)-dimensional space-time white noise ($f= \delta_0$) where $\gamma < \frac{\beta}{4}$ for \cite{salins2022global} and $\gamma<\frac{(1+\beta)}{6}$ for Theorem \ref{thm:main}. 
\end{remark}

\begin{remark}
\label{rmk: relation between beta and gamma}

 Note that choice of $(\beta,\gamma)$ in \eqref{conditions on beta and gamma} ensures that there exists $p\geq 2 $ such that $p \in \left( \frac{d+2}{\kappa}, \frac{1+\beta}{\gamma} \right)$. We briefly explain why this range of $p$ in \eqref{conditions on beta and gamma} is necessary for the proof of Theorem \ref{thm:main} with a sketch of the proof in the case of the one-dimensional SPDE with the space-time white noise and simple coefficients:
 \begin{equation}\label{eq:simple SPDE}
   du = u_{xx} - u^{1+\beta} + u^{1+\gamma} e_k dw_t^k,\quad t\in(0,T],\,\,x\in\bR; \quad u(0,\cdot) = u_0,
 \end{equation} where $T<\infty$, $\{e_k:k\in\bN\}$ is the orthonormal $L_2(\bR)$ basis, and $ \{w_t^k:k\in\bN\}$ is a set of one-dimensional independent Wiener process. In order to obtain the global existence of a solution to \eqref{eq:simple SPDE}, we need to have a uniform bound of a nonnegative local solutions $u_m$. For this purpose, we separate $u_m$ into two parts: the noise-dominant part $v_m$ and the dissipative-dominant part $u_m-v_m$. Let $v_m$ satisfy the following SPDE
 \begin{equation}
   \label{noise dominant equation}
dv_m = v_{mxx}  + u_m^{1+\gamma} e_k dw_t^k,\quad t\in(0,T],\,\,x\in\bR; \quad v_m(0,\cdot) = u_0.
 \end{equation} Note that the stochastic part of \eqref{noise dominant equation} is $u_m^{1+\gamma}$, not $v_m^{1+\gamma}$. Then, using Duhamel's principle we can deduce that $u_m\leq v_m$ since $u_m-v_m$ satisfies
\begin{equation*}
u_m(t,x) - v_m(t,x) = -\int_0^t\int_{\bR}p(t-s,x-y)u_m^{1+\beta}(s,y)dyds\leq 0,
\end{equation*}
where $p(t,x) = \frac{1}{(4\pi t)^{1/2}}e^{-\frac{1}{4t}|x-y|^2}$ denotes the heat kernel. Consequently, to control $\sup_{t,x}u_m(t,x)$, it suffices to control $\sup_{t,x}v_m(t,x)$. Additionally, since $v\in \cH_p^{\kappa}(\tau)$, Corollary \ref{embedding_corollary} yields that ${\bE}\sup_{t,x}|v_m(t,x)|^p$ is controlled by $\| v_m\|^p_{\cH^\kappa_p}$. Thus, we will control $\| v_m\|^p_{\cH^\kappa_p}$ in what follows. As $v_m$ is a solution to \eqref{noise dominant equation}, we get $\| v_m\|^p_{\cH^\kappa_p}\lesssim \|u_m \|^{p(1+\gamma)}_{\bL_{p(1+\gamma)}}$ (to simplify, we dropped the subscript for the stopping time). On the other hand, we have a uniform bound $\|u_m\|_{\bL_{1+\beta}}^{1+\beta}\lesssim 1$ {with the help of the conservation law} (see Lemma \ref{Lq bounds}). Therefore, if $p':=p(1+\gamma)-(1+\beta) < p$, then we can use Jensen's inequality and the fact $u_m\leq v_m$ to have 
\begin{equation*}
  \| v_m\|^p_{\cH^\kappa_p}\lesssim \|u_m \|^{p(1+\gamma)}_{\bL_{p(1+\gamma)}} \lesssim \|u_m\|_{\bL_{1+\beta}}^{1+\beta} \cdot \bE \sup_{t,x} |u_m|^{p'}\lesssim  \left( \bE \sup_{t,x}|v_m|^p \right)^{p'/p} \lesssim \| v_m \|^{p'/p}_{\cH^\kappa_p}.
\end{equation*} We now use Young's inequality to get a uniform bound of $\| v_m\|^p_{\cH^\kappa_p}$ until a suitable stopping time (see Lemma \ref{lem:noise related solution control}). Thus, $p(1+\gamma)-(1+\beta) < p$, (equivalent to $p<\frac{1+\beta}{\gamma}$ in \eqref{conditions on beta and gamma}), is required. The lower bound $p > \frac{d+2}{\kappa}$ is needed to choose $\alpha_1$ and $\alpha_2$ satisfying \eqref{condition for alpha1 and alpha 2} so that the H\"older exponents in \eqref{holder regularity of solution} are positive.

\end{remark}

By optimizing the choice of $(p,\alpha_1,\alpha_2)$ in \eqref{condition for alpha1 and alpha 2}, we can obtain the maximal H\"older regularity of the solution. However, it should be noted that the solution $u=u^{(p)}$  constructed in Theorem \ref{thm:main} depends on $p$. To obtain the maximal H\"older regularity, we let $p=p_\ep$ depend on a parameter $\ep>0$, which goes to $\frac{1+\beta}{\gamma}$ as $\ep\to 0,$ and all the conditions of Theorem \ref{thm:main} hold for all $p_\ep$.  To this end, we need to verify that $u^{(p_\ep)}$ for any $\ep>0$ coincides with each other, and $u=u^{(p_\ep)}$ is indeed a unique solution in $\cap_{\ep >0} \cH_{p_\ep,loc}^\kappa$. Thus, we need the following theorem.

\begin{theorem}
\label{uniqueness of solution in p}
Let $\beta,\gamma,p$ and $u_0$ satisfy the conditions in Theorem \ref{thm:main} and $u\in \cH_{p,loc}^\kappa$ be the solution to \eqref{main_equation}. Then, if $q > p$ satisfies 
\begin{equation*}
\frac{d+2}{\kappa} < p < q < \frac{1+ \beta}{\gamma}
\end{equation*}
and $u_0\in L_1(\Omega,\cF_0;L_1(\bR^d))\cap U_q^{\kappa}$, the solution $u$ also belongs to $\cH_{q,loc}^\kappa$.
\end{theorem}

\begin{corollary}
\label{Holder regularity of the solution}
Suppose {$u_0\in L_1(\Omega,\cF_0;L_1(\bR^d))\cap \left( \bigcap_{p\in\left( \frac{d+2}{\kappa},\frac{1+ \beta}{\gamma} \right)}U_p^{\kappa} \right)$}. Then, {we have \eqref{eq:Holder regularity of the solution}}. In other words, for any $T\in(0,\infty)$ and all sufficiently small $\ep>0$, almost suerly
\begin{equation*}
\sup_{t\leq T}|u(t,\cdot)|_{C^{\kappa-\frac{\gamma}{1+\beta}(d+2)-\ep}(\bR^d)}+\sup_{x\in\bR^d}|u(t,\cdot)|_{C^{\frac{1}{2}\left[ \kappa-\frac{\gamma}{1+\beta}(d+2) \right]-\ep}([0,T])}<\infty. 
\end{equation*}

\end{corollary}
\begin{proof} 
Let $\ep>0$ be a small but arbitrary constant that satisfies $\frac{1}{2}\left[ \kappa-\frac{\gamma}{1+\beta}(d+2) \right] > \ep$ and set $p_\ep = \frac{2(d+2)(1+\beta)}{2(d+2)\gamma+(1+\beta)\ep}$. Then, there exists a unique solution $u\in\cap_{\ep>0}\cH_{p_\ep,loc}^{\kappa}$ to \eqref{main_equation} due to Theorem \ref{thm:main} and Theorem \ref{uniqueness of solution in p}. 
Now choose $\alpha_1 = \frac{1}{p_{\ep}}+\frac{\ep}{8}$ and $\alpha_2 = \frac{1}{p_{\ep}} + \frac{\ep}{4}$. Then, by Theorem \ref{thm:main}, almost surely
\begin{equation*}
\sup_{t\in[0,T]}|u(t,\cdot)|_{C^{\kappa-\frac{\gamma}{1+\beta}(d+2)-\ep}}\leq |u|_{C^{\ep/8}([0,\tau];C^{\kappa-\frac{\gamma}{1+\beta}(d+2)-\ep}(\bR^d) )} < \infty.
\end{equation*}
In order to obtain the regularity in time, consider $p_{2\ep}$ instead of $p_\ep$ and choose $\alpha_1 = \frac{1}{2}\left( \kappa - \frac{d}{p_{2\ep}} \right)-\frac{\ep}{2}$ and $\alpha_2 = \frac{1}{2}\left( \kappa - \frac{d}{p_{2\ep}} \right)-\frac{\ep}{4}$. Then, Theorem \ref{thm:main} yields that almost surely
\begin{equation*}
\sup_{x\in\bR^d}|u(\cdot,x)|_{C^{\frac{1}{2}\left[ \kappa-(d+2)\frac{\gamma}{1+\beta} \right]-\ep}}\leq |u|_{C^{\ep/2}([0,\tau];C^{\frac{1}{2}\left[ \kappa-(d+2)\frac{\gamma}{1+\beta} \right]-\ep}(\bR^d) )} < \infty.
\end{equation*}
The corollary is proved.

\end{proof}

\vspace{2mm}

\section{Proof of Theorem \ref{thm:main}}\label{sec:proof of main thm}

This section devoted to proving Theorem \ref{thm:main}. The main challenge is to construct a global solution to \eqref{main_equation} as the the uniqueness and regularity of the global solution follow from those of local solutions. In Section \ref{subsec:Lp solvability of SPDE}, we establish the existence, uniqueness, and regularity of local solutions $u_m$ on some time interval $[0,\tau_m]$. The properties of local solutions follow from the $L_p$ theory for semilinear SPDEs, as described in \cite[Theorem 5.1]{kry1999analytic}. In order to construct a global solution, we need a uniform bound of local solutions $u_m$. In Section \ref{subsec:uniform bound on local solutions}, we decompose the local solution $u_m$ into two components: the noise-dominant part $v_m$ and the dissipative-dominant part $u_m - v_m$. The properties of noise-dominant part $v_m$ are suggested in Lemma \ref{lem:well-posedness of v_m}. As illustrated in Remark \ref{rmk: relation between beta and gamma}, we show that the local solution $u_m$ is bounded by noise-dominant part $v_m$; i.e., the dissipativity-dominant part $u_m - v_m$ is non-positive (see Lemma \ref{lem:upper bound of u_m by v_m}).  Then, $\sup_{t,x}|u_m(t,x)|$ is bounded by $\|v_m\|_{\cH_p^{\kappa}(\tau)}$ because $\sup_{t,x}|v_m(t,x)|$ is bounded by $\|v_m\|_{\cH_p^{\kappa}(\tau)}$ (see Lemma \ref{lem:well-posedness of v_m}). Additionally, observe that $\|v_m\|_{\cH_p^{\kappa}(\tau)}$ is controlled by $\|u_m\|_{\bL_{p(1+\gamma)}(\tau)}$ (see \eqref{eq:bound of vm}) and there is a uniform bound of $\|u_m\|_{\bL_{1+\beta}(\tau)}$ {by It\^o's formula} (see Lemma \ref{Lq bounds}). In conclusion, the noise-dominant part $v_m$ can be controlled by the dissipativity term (see Lemma \ref{lem:noise related solution control}). With this bound, we prove that the candidate of a global solution ($u(t,\cdot)=u_m(t,\cdot)$ for $t\in[0,\tau_m]$) does not explode in finite time (see Lemma \ref{lem:uniform bound of local solutions in probability} and Section \ref{subsec:proof of main thm}). In Section \ref{subsec:uniqueness of solution in p} we prove Theorem \ref{uniqueness of solution in p}. 

\vspace{0.3cm}

\subsection{Existence, uniqueness, and regularity of local solutions}\label{subsec:Lp solvability of SPDE}

Let $\tau\leq T$ be a bounded stopping time and $n\in (-2,-1)$. We review the $L_p$-solvability of the SPDE
\begin{equation} \label{nonlinear_equation}
du = \left(Lu + g(u)\right) dt + \sigma^k(u) dw_t^k,\quad 0< t \leq \tau;\quad u(0,\cdot) = u_0,
\end{equation}
where $Lu(t,x) := a^{ij}(t,x)u_{x^ix^j}(t,x) + b^i(t,x) u_{x^i}(t,x) + c(t,x)u(t,x)$ and $w_t^k$ is a one-dimensional Brownian motion. We present some assumptions on $g$ and $\sigma$.

\begin{assumption}[$\tau$] \label{assumption:f_and_g}
\begin{enumerate}[(i)]

\item The functions $g(t,x,u)$ and $\sigma^k(t,x,u)$ are $\cP\times\cB(\bR^d)\times\cB(\bR)$-measurable satisfying $g(t,x,0)\in \bH_{p}^{n}(\tau)$ and $\sigma(t,x,0) = \left( \sigma^1(t,x,0),\sigma^2(t,x,0),\dots \right) \in \bH_{p}^{n+1}(\tau,\ell_2).$

\item
For any $\ep>0$, there exists a constant $N_\ep$ such that
\begin{equation}
\label{conditions_on_f_and_g}
\| g(u) - g(v) \|_{\bH_{p}^n(\tau)} + \| \sigma(u) - \sigma(v) \|_{\bH_{p}^{n+1}(\tau,\ell_2)} \leq \ep\| u-v \|_{\bH_p^{n+2}(\tau)} + N_\ep \| u-v \|_{\bH_p^{n+1}(\tau)}
\end{equation}
for any $u,v\in \bH_p^{n+2}(\tau)$.

\end{enumerate}
\end{assumption}

\begin{theorem}
\label{theorem_nonlinear_case}
Suppose Assumptions \ref{assumption:coefficients} and \ref{assumption:f_and_g} ($\tau$) hold. Then, for any $p\geq2$ and $u_0\in U_p^{n+2}$, \eqref{nonlinear_equation} admits a unique solution $u\in\cH_p^{n+2}(\tau)$ such that
\begin{equation} 
\label{nonlinear_estimate}
\| u \|_{\cH_p^{n+2}(\tau)} \leq N\left(\|g(0)\|_{\bH_p^{n}(\tau)} + \|\sigma(0)\|_{\bH_p^{n+1}(\tau,\ell_2)} + \| u_0 \|_{U_p^{n+2}}\right),
\end{equation}
where $N = N(n,d,p,K,T,N_\ep)$. 
\end{theorem}
\begin{proof}
Since \cite[Theorem 5.1]{kry1999analytic} covers the case $\tau = T$, we only prove the case $\tau < T$. 

\textit{(Step 1). (Existence)}
 Set
$$
\bar{g}(t, u) := 1_{t\leq \tau} g(t,u)\quad\text{and}\quad \bar{\sigma}(t, u) := 1_{t\leq \tau} \sigma(t,u).
$$
Additionally, $\bar{g}(u)$ and $\bar{\sigma}(u)$ satisfy Assumption \ref{assumption:f_and_g} ($T$). Then, by \cite[Theorem 5.1]{kry1999analytic}, there exists a unique solution $u\in \cH_{p}^{n+2}(T)$ such that $u$ satisfies \eqref{nonlinear_equation} with $\bar g$ and $\bar{\sigma}$, instead of $g$ and $\sigma$, respectively. As $\tau\leq T$, we have $u\in\cH_{p}^{n+2}(\tau)$ and $u$ satisfies \eqref{nonlinear_equation} and the estimate \eqref{nonlinear_estimate} with $g$ and $\sigma$.

\textit{(Step 2). (Uniqueness)}
Let $u,v\in \cH_{p}^{n+2}(\tau)$ be two solutions of \eqref{nonlinear_equation}. Then, \cite[Theorem 5.1]{kry1999analytic} yields there exists a unique solution $\bar{v}\in \cH_{p}^{n+2}(T)$ satisfying
\begin{equation}
\label{equation_in_proof_of_uniqueness}
d \bar{v} = \left( a^{ij}\bar{v}_{x^ix^j} + b^{i}\bar{v}_{x^i} + c\bar{v} + \bar{g}(v) \right)dt  + \bar{\sigma}^k(v)  dw^k_t, \quad 0<t\leq T\, ; \quad \bar{v}(0,\cdot)=u_0.
\end{equation}
Notice that in \eqref{equation_in_proof_of_uniqueness}, $\bar{g}(v)$ and $\bar{\sigma}(v)$ are used instead of $\bar{g}(\bar{v})$ and $\bar{\sigma}(\bar{v})$, respectively.  Set  $\tilde v:=v-\bar{v}$. Then, for fixed $\omega\in\Omega$, we have
$$
d \tilde v = \left( a^{ij}\tilde{v}_{x^ix^j} + b^{i}\tilde{v}_{x^i} + c\tilde{v} \right) dt, \quad 0 < t \leq \tau\, ; \quad \tilde v(0,\cdot)=0.
$$
By the deterministic version of \cite[Theorem 5.1]{kry1999analytic}, we have $\tilde v = 0$ in $L_p((0,\tau]\times\bR^d)$ almost surely. Additionally, it implies $v(t,\cdot) = \bar{v}(t,\cdot)$ in $L_p((0,\tau]\times\bR^d)$ almost surely. Thus, in \eqref{equation_in_proof_of_uniqueness}, we can replace $\bar g(v)$ and $\bar{\sigma}(v)$ with $\bar g (\bar v)$ and $\bar{\sigma}(\bar{v})$. Therefore, $\bar v\in \cH_p^{n+2}(T)$ satisfies  \eqref{nonlinear_equation} on $(0,T]$ with $\bar g, \bar \sigma$ instead of $g,\sigma$, respectively. Similarly, by following word for word, there exists $\bar{u}\in \cH_p^{n+2}(T)$ such that $\bar{u}$ satisfies  \eqref{nonlinear_equation} on $(0,T]$ with $\bar g$ and $\bar \sigma$ instead of $g$ and $\sigma$.
Thus, by the uniqueness result in $\cH_{p}^{n+2}(T)$, we have $\bar{u} = \bar{v}$ in $\cH_{p}^{n+2}(T)$, which implies $u = v$ in $\cH_p^{n+2}(\tau)$. Thus, the lemma is proved.

\end{proof}

Subsequently, we proceed to establish the existence, uniqueness, and regularity of a local solution for equation \eqref{main_equation}. Take a nonnegative function $h(z)\in C^1(\bR)$ such that $h(z) = 1$ on $|z|\leq1$ and $h(z) = 0$ on $|z|\geq2$. For $m\in\bN$, set 
\begin{equation}
\label{def of hm}
h_m(z) := h(z/m),
\end{equation} and recall $Lu(t,x) := a^{ij}(t,x)u_{x^ix^j}(t,x) + b^i(t,x) u_{x^i}(t,x) + c(t,x)u(t,x)$.

\begin{lemma}  
\label{cut_off_lemma}
Let $\beta,\gamma\in(0,\infty)$, $T\in(0,\infty)$, $m\in \bN$, and $p\geq2$. Suppose Assumption \ref{assumption:coefficients} holds. Then, for a bounded stopping time $\tau\leq T$ and nonnegative initial data $u_0\in U_{p}^{\kappa}$, there exists a unique nonnegative solution $u_m \in \cH_p^{\kappa}(\tau)$ satisfying equation 
\begin{equation} 
\label{eq:cut_off_equation}
du_m = \left(Lu_m - \bar{b}u_{m+}^{1+\beta} h_m(u_m)  \right) dt + \xi u^{1+\gamma}_{m+} h_m(u_m) (f\ast e_k) dw_t^k,\quad 0 < t \leq \tau; \quad u_m(0,\cdot) = u_0(\cdot),
\end{equation}
where $e_k$ and $h_m$ are introduced in Remark \ref{remark:representation} and \eqref{def of hm}, respectively. Furthermore, for $\alpha_1,\alpha_2\in(0,\infty)$ and $p\geq2$ satisfying \eqref{condition_for_alpha_beta}, we have
$u_m \in C^{\alpha_1-1/p}([0,\tau];C^{\kappa-2\alpha_2-1/p}(\bR^d) )$ almost surely and 
\begin{equation*} 
\mathbb{E} \|u_m\|^p_{C^{\alpha_1-1/p}([0,\tau];C^{\kappa-2\alpha_2-1/p}(\bR^d) )}\leq N(\alpha_1,\alpha_2,\kappa,p,T)\|u_m\|_{\cH^{\kappa}_{p} (\tau)}^p.
\end{equation*}
\end{lemma}
\begin{proof}
Observe that for $u,v\in \bR$ and $\lambda>0$, we have
\begin{equation}
\label{nonlinear_cutoff}
I_\lambda(u,v) := \left| u_+^{1+\lambda}h_m(u) - v_+^{1+\lambda}h_m(v) \right| \leq N(\lambda,m)|u-v|.
\end{equation}
Indeed,
\begin{equation} \label{lipschitz_check_cut_off}
\begin{aligned}
&\left| u_+^{1+\lambda}h_m(u) - v_+^{1+\lambda}h_m(v) \right| \\
&\quad=
\begin{cases}
\,\,\left| u^{1+\lambda}h_m(u) - v^{1+\lambda}h_m(v) \right| \leq N_m|u-v|  &\text{ if}\quad u,v\geq0,\\
\,\,u^{1+\lambda}h_m(u) \leq (2m)^\lambda u \leq (2m)^\lambda (u-v) = N_m|u-v| &\text{ if}\quad u\geq0, v < 0,\\
\,\,v^{1+\lambda}h_m(v) \leq (2m)^\lambda v \leq (2m)^\lambda (v-u) = N_m|u-v| &\text{ if}\quad u<0, v\geq0,\\
\,\,0\leq N_m|u-v| &\text{ if}\quad u,v < 0.
\end{cases}
\end{aligned}
\end{equation}
Let $u,v\in \bH_p^{\kappa}(\tau)$. Then, by Lemmas \ref{prop_of_bessel_space} \eqref{norm_bounded} and \eqref{nonlinear_cutoff} with $\lambda = \beta$, we have
\begin{equation} 
\label{cut_off_lower_order_cal_1}
\begin{aligned}
&\left\| \bar b(t,\cdot) u_+^{1+\beta}(t,\cdot) h_m(u(t,\cdot)) - \bar b(t,\cdot) v^{1+\beta}_+(t,\cdot) h_m(v(t,\cdot)) \right\|^p_{H_p^{\kappa-2}} \\
&\quad = \int_{\bR^d}\left(\int_{\bR^d} R_{2-\kappa}(x-y) \bar b(t,y) I_\beta(u(t,y),v(t,y))  dy\right)^{p}dx \\
&\quad \leq \left( \int_{\bR^d}|R_{2-\kappa}(x)|dx \right)^p\int_{\bR^d}   \left| u^{1+\beta}_+(t,y) h_m(u(t,y)) - v^{1+\beta}_+(t,y) h_m(v(t,y)) \right|^p  dy \\
&\quad\leq N(\beta,\kappa,m,p) \| u(t,\cdot) - v(t,\cdot) \|_{L_p}^p
\end{aligned}
\end{equation}
on $(\omega,t)\in\opar0,\tau\cbrk$. Additionally, if we set $\bm{e} := (e_1,e_2,\dots)$, similarly to the proofs of \cite[Lemma 3]{ferrante2006spdes} and \cite[Lemma 4.4]{choi2021regularity}, on $(\omega,t)\in\opar0,\tau\cbrk$,
\begin{equation} \label{cut_off_lower_order_cal_2}
\begin{aligned}
&\left\| \xi(t,\cdot) \left[u^{1+\gamma}_+(t,\cdot) h_m(u(t,\cdot)) -  v^{1+\gamma}_+(t,\cdot) h_m(v(t,\cdot))\right] (f\ast \bm{e}) \right\|_{H^{\kappa-1}_p(\ell_2)}^p \\
&\quad\leq \int_{\bR^d}\bigg( \int_{\bR^d}\int_{\bR^d} R_{1-\kappa}(x-(y-z))R_{1-\kappa}(x-z) \\
&\quad\quad \times\left[\xi(t,y-z)\xi(t,-z)
I_\gamma(u(t,y-z),v(t,y-z))
I_\gamma( u(t,-z), v(t,-z)) \right]
dzf(dy)\bigg)^{p/2} dx \\
& \quad\leq \nu_{\kappa}^{p/2}\int_{\bR^d}\left|  u^{1+\gamma}_+(t,\cdot) h_m(u(t,\cdot)) -  v_+^{1+\gamma}(t,\cdot) h_m(v(t,\cdot))\right|^p dx\\
&\quad\leq N(\gamma,\kappa,m,p) \| u(t,\cdot) - v(t,\cdot) \|^p_{L_p},
\end{aligned}
\end{equation}
where $\nu_{\kappa}$ is the constant introduced in \eqref{def of nu kappa}. By taking integration with respect to $(\omega,t)$ and Lemma \ref{prop_of_bessel_space} \eqref{multi_ineq}, we have
\begin{equation} \label{estimate_for_cut_off_equation}
\begin{aligned}
&\left\| \bar b \left( u_+^{1+\beta} h_m(u) -  v_+^{1+\beta} h_m(v) \right) \right\|^p_{\bH_p^{\kappa-2}(\tau)} + \left\| \xi \left( u_+^{1+\gamma} h_m(u) -  v_+^{1+\gamma} h_m(v) \right)\bm{e} \right\|_{H^{\kappa-1}_p(\ell_2)}^p \\
&\quad\leq N(\beta,\gamma,\kappa,m,p) \| u-v \|_{\bL_p(\tau)}^p \\
&\quad\leq \ep \| u-v \|_{\bH^{\kappa}_p(\tau)}^p +  N(\beta,\gamma,\kappa,m,p,\ep) \| u-v \|_{\bH^{\kappa-1}_p(\tau)}^p,
\end{aligned} 
\end{equation} where we used Young's inequality in the last inequality.
Therefore, by Theorem \ref{theorem_nonlinear_case}, there exists a unique solution $u_m\in \cH_p^{\kappa}(\tau)$ to \eqref{eq:cut_off_equation}.

Next, we prove $u_m\geq0$. By considering mollification and multiplying cut-off functions, there exists a sequence of functions $\{u^n_0 \in  U_p^{1}:u^n_0\geq 0,n\in\bN\}$ such that $u^n_0 \to u_0$ in $U^{\kappa}_{p}$. By Theorem \ref{theorem_nonlinear_case}, there exists a unique solution  $u^n_m\in \cH_p^1(\tau)$ satisfying 
\begin{equation*}
du^n_m = \left(Lu_m^n - \bar{b} (u^n_{m+})^{1+\beta} h_m(u_m^n)  \right) dt + \sum_{k = 1}^n \xi (u_{m+}^n)^{1+\gamma}h_m(u_m^n)(f\ast e_k) dw_t^k,\quad 0 < t \leq \tau
\end{equation*}
with $u_m^n(0,\cdot) = u^n_0(\cdot)$. If we set $\bm{e}_n := (e_1,\dots,e_n,0,0,\dots)$, by Theorem \ref{theorem_nonlinear_case} and a similar computation with \eqref{estimate_for_cut_off_equation}, we have
\begin{equation*}
\begin{aligned}
&\|u_m - u_m^n\|_{\cH_p^{\kappa}(\tau\wedge t)}^p \\
&\quad\leq N\|u_0 - u_0^n\|_{U_p^{\kappa}}^p + N\left\| \bar{b}\left( (u_{m+})^{1+\beta} h_m(u_m) - (u_{m+}^n)^{1+\beta} h_m(u_{m}^n) \right) \right\|_{\bH_p^{\kappa-2}(\tau\wedge t)}^p \\
&\quad\quad + N\left\| \xi((u_{m+})^{1+\gamma}f\ast\bm{e} - (u_{m+}^n)^{1+\gamma}f\ast\bm{e}_n) \right\|_{\bH^{\kappa-1}_p(\tau\wedge t,\ell_2)}^p \\
&\quad\leq N\|u_0 - u_0^n\|_{U_p^{\kappa}}^p + N_m \| u_{m} - u_{m}^n \|_{\bL_p(\tau\wedge t)}^p + N\left\| \xi (u_{m+})^{1+\gamma} (f\ast\bm{e} - f\ast\bm{e}_n) \right\|_{\bH^{\kappa-1}_p(\tau\wedge t,\ell_2)}^p ,
\end{aligned}
\end{equation*}
where $N = N(\kappa,d,m,p,K,T)$. By Theorem \ref{embedding} \eqref{gronwall_type_ineq}, we have
\begin{equation}\label{eq:approximation by u^n_m}
\|u_m - u_m^n\|_{\cH_p^{\kappa}(\tau\wedge t)}^p \leq N \|u_0 - u_0^n\|_{U_p^{\kappa}}^p + N\left\| \xi (u_{m+})^{1+\gamma} (f\ast\bm{e} - f\ast\bm{e}_n) \right\|_{\bH^{\kappa-1}_p(\tau\wedge t,\ell_2)}^p,
\end{equation}
where $N = N(\kappa,m,d,p,K,T)$. Note that
\begin{equation*}
\begin{aligned}
&\left\| \xi (u_{m+})^{1+\gamma} (f\ast\bm{e} - f\ast\bm{e}_n) \right\|^p_{H_{p}^{\kappa-1}(\ell_2)} \\
&\quad \leq N  \int_{\bR} \left( \sum_{k>n} \left[ \int_{\bR}R_{1-\kappa}(x-y) \xi(t,y)(u_{m+}(t,y))^{1+\gamma} (f\ast e_k)(y) dy \right]^2 \right)^{p/2} dx.
\end{aligned}
\end{equation*}
Moreover, we have
\begin{equation*}
\begin{aligned}
&\sum_{k>n} \left[ \int_{\bR}R_{1-\kappa}(x-y) \xi(t,y)u_{m+}^{1+\gamma}(t,y) (f\ast e_k)(y) dy \right]^2 \\
&\quad \leq \sum_{k} \left[ \int_{\bR}R_{1-\kappa}(x-y) \xi(t,y)(u_{m+}(t,y))^{1+\gamma} (f\ast e_k)(y) dy \right]^2 \\
&\quad \leq  K^2\int_{\bR}|R_{1-\kappa}(x-y)|^2 |u_{m+}(t,y)|^{2+2\gamma}  dy.
\end{aligned}
\end{equation*}
Thus, by the dominated convergence theorem, we have
\begin{equation*}
\left\| \xi(t,\cdot) (u_{m+}(t,\cdot))^{1+\gamma}(f\ast\bm{e} - f\ast\bm{e}_n)(\cdot) \right\|^p_{H_{p}^{\kappa-1}(\ell_2)} \to 0 \quad \text{as}\quad n\to \infty,
\end{equation*}
for almost every $(\omega,t)$. Then, by using the dominated convergence theorem again, we have
\begin{equation*}
\left\| \xi(t,\cdot) (u_{m+}(t,\cdot))^{1+\gamma}(f\ast\bm{e} - f\ast\bm{e}_n)(\cdot) \right\|^p_{\bH_{p}^{\kappa-1}(\tau,\ell_2)} \to 0 \quad \text{as}\quad n\to \infty.
\end{equation*}
By \eqref{eq:approximation by u^n_m}, we get $u^n_m \to u_m$ in probability as $n\to \infty$. Since we know that the limit $u_m$ is unique, it suffices to show that $u^n_{m}\geq0$ for each $n$. Since $u^n_{m}\in\cH_p^1(\tau)$,  By applying \cite[Theorem 2.5]{krylov2007maximum} with $f = -\bar b (u^n_{m+})^{1+\beta}h_m(u^n_m)$ and $g = \xi(u^n_{m+})^{1+\gamma}h_m(u^n_m)(f\ast\bm{e}_n)$, we have $u^n_{m}\geq0$. The lemma is proved.
\end{proof}

\subsection{Uniform bounds on local solutions}\label{subsec:uniform bound on local solutions}

In this subsection, we prove that the local solution does not experience a finite time explosion (see Lemma \ref{lem:uniform bound of local solutions in probability}).

The following lemma is useful to control the coefficients of \eqref{main_equation}.

\begin{lemma}[Lemma 5.5 of \cite{choi2021regularity}]\label{lem:cefficient control}
For $k\in \bN$, define
$$\psi_k(x) := \frac{1}{\cosh(|x|/k)}.
$$
\begin{enumerate}[(i)]

\item For $i,j\in\{1,2,\dots,d\}$,
\begin{equation}\label{eq:derivatives of psi}
  |\psi_{kx^i}(x)| \leq k^{-1}\psi_k(x)\quad \text{and} \quad|\psi_{kx^ix^j}(x)| \leq N(d) k^{-2} \psi_{k}(x).
\end{equation}

\item 
Suppose Assumption \ref{assumption:coefficients} holds. Then, for $(\omega,t,x)\in\Omega\times(0,\infty)\times\bR$, we have
\begin{equation}
\label{eq:coefficient control}
\begin{aligned}
& (a^{ij}\psi_k)_{x^ix^j} - (b^{i}\psi_k)_{x^i} + (c - K)\psi_k\\
& = a^{ij}\psi_{kx^ix^j} + \left( 2a^{ij}_{x^j} - b^{i} \right)\psi_{kx^i} + \left( a^{ij}_{x^ix^j} - b^{i}_{x^i} + c - K \right)\psi_{k} \leq 0.
\end{aligned}
\end{equation}
\end{enumerate}

\end{lemma}

With the help of Lemma \ref{lem:cefficient control}, we obtain a uniform $L_{1+\beta}$-bound for local solutions.

\begin{lemma}
\label{Lq bounds}
Let $\tau\leq T$ be a bounded stopping time and $m\in \bN$ be fixed. Suppose all the assumptions of Lemma \ref{cut_off_lemma} hold and we assume that $u_0\in  L_1(\Omega,\cF_0;L_1(\bR^d))\cap U_{p}^{\kappa}$. If $u_m\in \cH_p^{\kappa}(\tau)$ is the solution to \eqref{cut_off_lemma} introduced in Lemma \ref{cut_off_lemma}, then
\begin{equation} 
\label{estimate of L1+beta bound}
\bE \int_0^\tau \int_{\bR} (u_{m}(s,x))_+^{1+\beta}h_m(u_m(s,x)) dx ds \leq  N(K,T) \bE \|u_0\|_{L_1},
\end{equation}
where $h_m$ is the function introduced in \eqref{def of hm}.
\end{lemma}
\begin{proof}
Let $m\in \bN$ be fixed. For $n\in\bN$, let $\psi_n$ be a function given in Lemma \ref{lem:cefficient control}. Recall that $u_m$ is the solution to \eqref{eq:cut_off_equation}. Since $\psi_n\in \cS$, by employing It\^o's formula, integration by parts, and \eqref{eq:coefficient control}, we have
\begin{equation}
\label{L1_bound_calculation_1}
\begin{aligned}
& e^{-4K\tau}\left(u_m(\tau ,\cdot),\psi_n  \right)\\
&\leq (u_0,\psi_n)  - \int_0^\tau\int_{\bR} \bar b (s,x)  u_{m+}^{1+\beta}(s,x)  h_m(u_m(s,x))\psi_n(x)e^{-4Ks}dxds  \\
&\quad + \int_0^\tau\int_{\bR} \xi(s,x)u_{m+}^{1+\gamma}(s,x)h_m(u_m(s,x))(f\ast e_k)(x)\psi_n(x)dxdw_t^k.
\end{aligned}
\end{equation}
By taking expectations
we have
\begin{equation}
\label{L1_bound_calculation_2}
\begin{aligned}
&e^{-4KT}\bE(u_m(\tau,\cdot),\psi_n) + K^{-1}e^{-4KT}\bE\int_0^\tau\int_{\bR}  u^{1+\beta}_m(s,x)  h_m(u_m(s,x)) \psi_n(x) dxds
\leq \bE\int_\bR u_0(x)dx.
\end{aligned}
\end{equation} 
By letting $n\to\infty$, we have \eqref{estimate of L1+beta bound} due to the nonnegativity of $u_m$.

\end{proof}

{The following three lemmas (Lemmas \ref{lem:well-posedness of v_m}, \ref{lem:upper bound of u_m by v_m}, and \ref{lem:noise related solution control}) are used to get an upper bound of supremum of local solutions $u_m$ (see \eqref{bound of vm - 2} in Lemma \ref{lem:noise related solution control}).

}

\begin{lemma}
\label{lem:well-posedness of v_m}
Suppose all the assumptions of Lemma \ref{cut_off_lemma} hold and $u_m\in \cH_{p}^{\kappa}(\tau)$ is the solution to \eqref{eq:cut_off_equation} introduced in Lemma \ref{cut_off_lemma}. Then, there exists a unique solution $v_{m}\in\cH_{p}^{\kappa}(\tau)$ such that
\begin{equation}
\label{eq:noise_dominating_equation}
dv_m = Lv_m dt + \xi u^{1+\gamma}_{m+}h_m(u_m)(f\ast e_k) dw_t^k,\quad 0 < t \leq \tau;\quad v_m(0,\cdot) = u_0.
\end{equation}
Furthermore, for $\alpha_1,\alpha_2\in(0,\infty)$ and $p\geq2$ satisfying \eqref{condition_for_alpha_beta}, we have
$$v_m \in C^{\alpha_1-1/p}([0,\tau];C^{\kappa-2\alpha_2-1/p}(\bR^d) )$$ 
almost surely and 
\begin{equation} 
\label{embedding for vm}
\mathbb{E} |v_m|^p_{C^{\alpha_1-1/p}([0,\tau];C^{\kappa-2\alpha_2-1/p}(\bR^d) )}\leq N(\alpha_1,\alpha_2,\kappa,p,K,T)\|v_m\|_{\cH^{\kappa}_{p} (\tau)}^p.
\end{equation}
\end{lemma}
\begin{proof}

Letting $u = u_m$ and $v = 0$ in \eqref{cut_off_lower_order_cal_2}, we have $\xi u_{m+}^{1+\gamma}h_m(u_m)(f\ast \bm{e}) \in \bH_p^{\kappa-1}(\tau,\ell_2)$ through a same computation as in $\eqref{cut_off_lower_order_cal_2}${, where $\bm{e} := (e_1,e_2,\dots)$}. This shows that \eqref{eq:noise_dominating_equation} satisfies Assumption \ref{assumption:f_and_g}. As a result, the application of Theorem \ref{theorem_nonlinear_case} and Corollary \ref{embedding_corollary} produces the desired result.
\end{proof}

\begin{lemma}
\label{lem:upper bound of u_m by v_m}
Let all the assumptions of Lemma \ref{cut_off_lemma} hold. Suppose $u_m\in \cH_p^{\kappa}(\tau)$ and $v_m\in \cH_p^{\kappa}(\tau)$ be the solutions introduced in Lemmas \ref{cut_off_lemma} and \ref{lem:well-posedness of v_m}, respectively. Then,  we have $u_m \leq v_m$.
\end{lemma}
\begin{proof}

For convenience, set $w_m(t,x):= u_m(t,x) - v_m(t,x)$. Choose a nonnegative function $\zeta\in C_c^\infty(\bR^d)$ such that $\int_{\bR^d} \zeta(x) dx = 1$ and define $\zeta_{\ep}(x) := \ep^{-d}\zeta(x/\ep)$. Let $\phi^{(\ep)}(x) := \int_{\bR^d} \phi(y) \zeta_\ep(x-y)dy$ for any function $\phi\in L_p(\bR^d)$. Then, by considering $\zeta_\ep$ as test functions to \eqref{eq:cut_off_equation} and \eqref{eq:noise_dominating_equation}, we have
\begin{equation}
\label{mollified eq}
\begin{aligned}
w^{(\ep)}_m(t,x) 
&= \int_{0}^{t} \left(a^{ij}_{x^ix^j}w_m\right)^{(\ep)}(s,x) - 2 \left(a^{ij}_{x^j}w_m\right)^{(\ep)}_{x^j}(s,x) + \left(a^{ij}w_m\right)^{(\ep)}_{x^ix^j}(s,x) - \left(b^i_{x^i}w_m\right)^{(\ep)}(s,x) ds \\
&\quad + \int_0^{t} \left(b^i w_m\right)_{x^i}^{(\ep)}(s,x) + \left(cw_m\right)^{(\ep)}(s,x) - \left( \bar b u^{1+\beta}_{m+}(s,\cdot) h_m(u_m(s,\cdot)) \right)^{(\ep)}(x) ds
\end{aligned}
\end{equation}
holds for all $t\leq \tau$ almost surely.

Next, for $l = 0,1,2,\dots,$ define
$ \delta_l := \exp\left( \frac{-l(l+1)}{2} \right),
$
and a nonnegative function $\varphi_l\in C_c^\infty(\bR)$ such that 
\begin{equation}
\label{prop of varphi}
\varphi_l(z) \leq \frac{2}{lz}\quad\text{on}\quad (\delta_l,\delta_{l-1}),\quad\varphi_l = 0\quad\text{on}\quad(\delta_l,\delta_{l-1})^c,
\end{equation}
and $\int_{\delta_l}^{\delta_{l-1}}\varphi_l(z)dz = 1$.
Set
\begin{equation}
\label{def of rn}
r_l(z) = \int_0^z\int_0^{z_2} \varphi_l(z_1) dz_1 dz_2.
\end{equation}
Observe that $r_l(z)$ is twice continuously differentiable, 
\begin{equation}
\label{prop of r'}
1_{[\delta_{l-1},\infty)}(z) \leq r_l'(z) \leq 1_{[\delta_l,\infty)}(z),
\end{equation}
\newline
and 
\begin{equation}
\label{prop of r}
(z - \delta_{l-1})1_{[\delta_{l-1},\infty)} \leq r_l(z) \leq (z - \delta_{l})1_{[\delta_{l},\infty)}.
\end{equation}
Let $q>0$ be specified later. Then, for almost sure $\omega\in\Omega$, by multiplying $\psi_k$ introduced in Lemma \ref{lem:cefficient control}, taking the integration, and the chain rule, we have
\begin{equation}\label{eq:integration by parts for r_lpsi_k}
  \int_{\bR^d}r_l\left( w_m^{(\ep)}(t,x) \right)\psi_k(x) e^{-qt}dx  = ({\bf A}) + ({\bf B}) +({\bf C}),
\end{equation} where 
\begin{equation}\label{eq:definition of A,B}
\begin{aligned}
   ({\bf A})&:= \int_0^t \int_{\bR^d} r_l'\left( w_m^{(\ep)}(s,x) \right) \left(a^{ij}w_m\right)^{(\ep)}_{x^ix^j}(s,x) \psi_k(x) e^{-qs}dxds ,\\
   ({\bf B})&:= \int_0^t \int_{\bR^d} r_l'\left( w_m^{(\ep)}(s,x) \right) \left[ - 2 \left(a^{ij}_{x^j}w_m\right)^{(\ep)}_{x^i}(s,x) + \left(b^iw_m\right)_{x^i}^{(\ep)}(s,x) \right] \psi_k(x)e^{-qs}dxds,
\end{aligned}
\end{equation} and 
\begin{equation}\label{eq:definition of C}
  \begin{aligned}
     ({\bf C})&:=\int_0^t \int_{\bR^d} r_l'\left( w_m^{(\ep)}(s,x) \right) \left( \left(  (a^{ij}_{x^ix^j} - b^i_{x^i} + c)(s,\cdot)  \right)w_m(s,\cdot) \right)^{(\ep)}(x)\psi_k(x)e^{-qs} dxds \\
&\quad - \int_0^t \int_{\bR^d} r_l'\left( w_m^{(\ep)}(s,x) \right) \left(\bar b(s,\cdot) u^{1+\beta}_{m+}(s,\cdot) h_m(u_m(s,\cdot)) \right)^{(\ep)}(x)\psi_k(x) e^{-qs}dxds \\
&\quad -q\int_0^t \int_{\bR^d} r_l\left( w_m^{(\ep)}(s,x) \right)  \psi_k(x)e^{-qs} dxds.
  \end{aligned}
\end{equation}
We first bound $({\bf A})$. Observe that since $a^{ij}(t,x)$ is differentiable in $x$, we have 
\begin{equation}\label{eq:approximation of the difference of leading terms}
  \begin{aligned}
& \left| \left( a^{ij}w_{m} \right)_{x^j}^{(\ep)}(s,x) - a^{ij}(s,x)w_{mx^j}^{(\ep)}(s,x) \right| \\
&\quad = \ep^{-1}\left|\int_{\bR^d}  \left(  a^{ij}(s,x-\ep y) - a^{ij}(s,x)  \right)w_m(s,x-\ep y)\zeta_{y^j}(y)  dy \right|\\
&\quad \leq N(K)\int_{\bR^d} |w_m(s,x-\ep y)||y\zeta_{y^j}(y)|dy.
\end{aligned} 
\end{equation} Using the integration by parts, we have 
\begin{equation}\label{eq:integration by parts for A}
  \begin{aligned}
    ({\bf A}) &= \int_{\bR^d} r'_l\left( w_m^{(\ep)}(s,x) \right) \left( a^{ij}w_m \right)_{x^ix^j}^{(\ep)}(s,x) \psi_k(x)dx \\
& = -\int_{\bR^d} r''_l\left( w_m^{(\ep)}(s,x) \right)  w_{mx^i}^{(\ep)}(s,x)\left( a^{ij}w_m \right)_{x^j}^{(\ep)}(s,x)\psi_k(x) dx \\
& \quad -\int_{\bR^d} r'_l\left( w_m^{(\ep)}(s,x) \right)  \left( a^{ij}w_m \right)_{x^j}^{(\ep)}(s,x) \psi_{kx^i}(x)dx \\
& = -\int_{\bR^d} r_l''\left(w_m^{(\ep)}(s,x)\right)  w_{mx^i}^{(\ep)}(s,x) \left[ \left(a^{ij}w_{m}\right)_{x^j}^{(\ep)}(s,x) - a^{ij}(s,x)w_{mx^j}^{(\ep)}(s,x) \right]\psi_k(x) dx \\
&\quad- \int_{\bR^d} a^{ij}(s,x) r_l''\left(w_m^{(\ep)}(s,x)\right)  w_{mx^i}^{(\ep)}(s,x) w_{mx^j}^{(\ep)}(s,x)  \psi_k(x)dx \\
& \quad -\int_{\bR^d} r'_l\left( w_m^{(\ep)}(s,x) \right)  \left[ \left( a^{ij}w_m \right)_{x^j}^{(\ep)}(s,x) - a^{ij}(s,x)w_{mx^j}^{(\ep)}(s,x) \right] \psi_{kx^i}(x)dx \\
& \quad -\int_{\bR^d} r'_l\left( w_m^{(\ep)}(s,x) \right)     a^{ij}(s,x)w_{mx^j}^{(\ep)}(s,x)  \psi_{kx^i}(x)dx. 
  \end{aligned}
\end{equation} Applying \eqref{eq:approximation of the difference of leading terms} and Young's inequality, we have 
\begin{equation}\label{eq:bound of A}
  \begin{aligned}
    ({\bf A}) &\leq N\int_{\bR^d} r_l''\left(w_m^{(\ep)}(s,x)\right)  \left| w_{mx^i}^{(\ep)}(s,x) \right|\int_{\bR^d} |w_m(s,x-\ep y)||y\zeta_{y^j}(y)|dy\psi_k(x)dx\\
&\quad - K^{-1}\int_{\bR^d} r_l''\left( w_m^{(\ep)}(s,x) \right)\left| w_{mx^j}^{(\ep)}(s,x) \right|^2 \psi_k(x)dx \\
& \quad + \int_{\bR^d} r_l'\left( w_m^{(\ep)}(s,x) \right)\int_{\bR^d}|w_m(s,x-\ep y)||y\zeta_{y^j}(y)|dy\psi_{kx^i}(x)dx \\
& \quad + \int_{\bR^d} r_l\left( w_m^{(\ep)}(s,x) \right) \left[ a_{x^j}^{ij}(s,x)\psi_{kx^i}(x) + a^{ij}(s,x)\psi_{kx^ix^j}(x) \right] dx \\
&\leq N\int_{\bR^d} r_l''\left(w_m^{(\ep)}(s,x)\right)  \left( \int_{\bR^d} |w_m(s,x-\ep y)||y\zeta_{y^i}(y)|dy \right)^2\psi_k(x)dx\\
&\quad - \frac{1}{2}K^{-1}\int_{\bR^d} r_l''\left( w_m^{(\ep)}(s,x) \right)\left| w_{mx^i}^{(\ep)}(s,x) \right|^2 \psi_k(x)dx \\
& \quad + k^{-1}\int_{\bR^d} r_l'\left( w_m^{(\ep)}(s,x) \right)\int_{\bR^d}|w_m(s,x-\ep y)||y\zeta_{y^j}(y)|dy\psi_{k}(x)dx \\
& \quad + k^{-1}\int_{\bR^d} r_l\left( w_m^{(\ep)}(s,x) \right) \psi_k(x) dx,
  \end{aligned}
\end{equation} where $N=N(K)>0$ and we have used \eqref{eq:derivatives of psi}. Now we estimate $({\bf B})$ using again the integration by parts, Young's inequality, and \eqref{eq:derivatives of psi} as follows:
\begin{equation}\label{eq:bound of B}
  \begin{aligned}
    ({\bf B}) &=   \int_{\bR^d} r_l''\left( w_m^{(\ep)}(s,x) \right) w_{mx^i}^{(\ep)}(s,x)\left[  2 \left(a^{ij}_{x^j}w_m\right)^{(\ep)}(s,x) - \left(b^i w_m\right)^{(\ep)}(s,x) \right] \psi_k(x)dx \\
& \quad + \int_{\bR^d} r_l'\left( w_m^{(\ep)}(s,x) \right) \left[  2 \left(a^{ij}_{x^j}w_m\right)^{(\ep)}(s,x) - \left(b^i w_m\right)^{(\ep)}(s,x) \right]\psi_{kx^i}(x) dx \\
& \leq N\int_{\bR^d} r_l''\left( w_m^{(\ep)}(s,x) \right) \left( \left(|w_m(s,\cdot)|\right)^{(\ep)}(x) \right)^2 \psi_k(x)dx \\
& \quad + \frac{1}{4}K^{-1}\int_{\bR^d} r_l''\left( w_m^{(\ep)}(s,x) \right) \left| w_{mx^i}^{(\ep)}(s,x) \right|^2 \psi_k(x)dx \\
& \quad + k^{-1}N\int_{\bR^d} r_l'\left( w_m^{(\ep)}(s,x) \right) \left(|w_m(s,\cdot)|\right)^{(\ep)}(x) \psi_{k}(x) dx,
  \end{aligned}
\end{equation} where we also used \eqref{boundedness_of_deterministic_coefficients} and $N=N(K)>0$. Since each integral in $({\bf C})$ is nonnegative (see the definition of $r_l$ in \eqref{def of rn}), we see that 
\begin{equation}\label{eq:bound of C}
   \begin{aligned}
     ({\bf C})&\leq N \int_0^t\int_{\bR^d} r_l'\left( w_m^{(\ep)}(s,x) \right)(|w_m(s,\cdot)|)^{(\ep)}(x) \psi_k(x) e^{-qs} dxds \\
&\quad -q\int_0^t \int_{\bR^d} r_l\left( w_m^{(\ep)}(s,x) \right)  \psi_k(x)e^{-qs} dxds.
   \end{aligned}
 \end{equation} Collecting all bounds on $({\bf A}),({\bf B})$ and $({\bf C})$ 
 in \eqref{eq:bound of A}--\eqref{eq:bound of C}, we obtain that 
\begin{equation}
\begin{aligned}
\int_{\bR^d}r_l&\left( w_m^{(\ep)}(t,x) \right) \psi_k(x)e^{-qt}dx \\
&\leq N\sum_i\int_0^t\int_{\bR^d} r''_l\left( w_m^{(\ep)}(s,x) \right)  \left( \int_{\bR^d} |w_m(s,x-\ep y)||y\zeta_{y^i}(y)|dy \right)^2\psi_k(x)e^{-qs}dx ds\\
&\quad + N\int_0^t\int_{\bR^d} r''_l\left( w_m^{(\ep)}(s,x) \right) \left( \left(|w_m(s,\cdot)|\right)^{(\ep)}(x) \right)^2\psi_k(x) e^{-qs}dx  ds \\
&\quad + k^{-1}\sum_i \int_0^t\int_{\bR^d} r_l'\left( w_m^{(\ep)}(s,x) \right)\int_{\bR^d}|w_m(s,x-\ep y)||y\zeta_{y^i}(y)|dy \psi_k(x) e^{-qs}dxds \\
&\quad + k^{-1}\int_0^t\int_{\bR^d} r_l\left( w_m^{(\ep)}(s,x) \right) \psi_k(x) e^{-qs} dxds \\
&\quad +(k^{-1}+1)N \int_0^t\int_{\bR^d} r_l'\left( w_m^{(\ep)}(s,x) \right)(|w_m(s,\cdot)|)^{(\ep)}(x) \psi_k(x) e^{-qs} dxds \\
&\quad -q\int_0^t \int_{\bR^d} r_l\left( w_m^{(\ep)}(s,x) \right)  \psi_k(x)e^{-qs} dxds.
\end{aligned} 
\end{equation} Letting $\ep\downarrow 0$, by \eqref{prop of varphi}--\eqref{prop of r}, we have 
\begin{equation}
  \begin{aligned}
    &\int_{\bR^d}r_l(w_m(t,x)) \psi_k(x)e^{-qt}dx \\
&\quad\leq N\int_0^t\int_{\bR^d} r''_l(w_m(s,x)) w_m^2(s,x) \psi_k(x) e^{-qs} dxds \\
&\quad\quad + k^{-1}\int_0^t\int_{\bR^d} r_l(w_m(s,x)) \psi_k(x) e^{-qs} dxds \\
&\quad\quad + (k^{-1}+1)N\int_0^t\int_{\bR^d} r_l'(w_m(s,x))w_m(s,x)\psi_k(x) e^{-qs}dxds \\
&\quad\quad - qN\int_0^t\int_{\bR^d} r_l(w_m(s,x))\psi_k(x) e^{-qs}dxds \\
&\quad\leq (N-q)\int_0^t\int_{\bR^d}  w_m(s,x) \psi_k(x) e^{-qs} dxds + q\delta_{l-1}\int_0^t\int_{\bR^d} \psi_k(x) e^{-qs}dxds.
  \end{aligned}
\end{equation} Now choose a large $q$ such that $N-q<0$. Then by letting $l\rightarrow\infty$ and $k\rightarrow \infty$ in order, we have 
\begin{equation*}
\begin{aligned}
\int_{\bR^d} w_{m+}(t,x)dx e^{-qt} \leq 0. 
\end{aligned}
\end{equation*}
Therefore, $w_{m+} = 0$ for all $t\leq \tau$ almost surely and this implies $u_m \leq v_m$. This completes the proof.

\end{proof}

Let all the assumptions of Lemma \ref{cut_off_lemma} hold and $u_m\in \cH_p^{\kappa}(\tau)$ be the solution to \eqref{cut_off_lemma} introduced in Lemma \ref{cut_off_lemma}. Then, {with the help of Lemma \ref{Lq bounds},} we can define a stopping time
\begin{equation*}
\begin{aligned}
&\tau_m(S) := \inf\left\{ t \leq \tau :  \int_0^t\int_{\bR}\left| u_m(s,x) \right |^{1+\beta}h_m(u_m(s,x))dxds\geq S \right\}.
\end{aligned}
\end{equation*}

\begin{lemma}
\label{lem:noise related solution control}
Let $\tau\leq T$. Assume $\beta>0$ and $\gamma>0$ satisfy \eqref{conditions on beta and gamma}. For $p\in\left(\frac{d+2}{\kappa},\frac{1+\beta}{\gamma}\right)$, $v_m \in \cH_p^{\kappa}(\tau)$ is the solutions to \eqref{eq:noise_dominating_equation} introduced in Lemma \ref{lem:well-posedness of v_m}. Then, we have
\begin{equation}
\label{bound of vm - 2}
\bE\sup_{t\leq \tau_m(S)}\sup_{x\in\bR}|u_m(t,x)|\leq \bE\sup_{t\leq \tau_m(S)}\sup_{x\in\bR}|v_m(t,x)| \leq N(\beta,\gamma,n,d,p,K,T,u_0,S).
\end{equation}
\end{lemma}
\begin{proof}
Set 
\begin{equation}
\label{def of p0}
p_0 := \frac{p}{p(\gamma+1)-(1+\beta)}
\end{equation}
and notice that $p_0 >1$ since $p < \frac{1+\beta}{\gamma}$. We apply Theorem \ref{theorem_nonlinear_case} with the solution $v_m \in \cH_p^{\kappa}(\tau)$ in Lemma \ref{lem:well-posedness of v_m} to obtain 
\begin{equation}
\label{eq:bound of vm}
\begin{aligned}
\| v_m \|_{\cH_p^{\kappa}(\tau_m(S))}^p - N\| u_0 \|_{U_p^{\kappa}}^p
&\leq N\bE\int_0^{\tau_m(S)}\left\| \xi u^{1+\gamma}(t,\cdot)h_m(u_m(t,\cdot))(f\ast\bm{e}) \right\|_{H_p^{\kappa-1}}^p dt \\
&\leq \nu_{\kappa}^{p/2}N\bE\int_0^{\tau_m(S)}\int_{\bR^d}\left|  u^{1+\gamma}_+(t,\cdot) h_m(u(t,\cdot)) \right|^p dxdt, 
\end{aligned}
\end{equation} where the last inequality follows from a similar computation as in \eqref{cut_off_lower_order_cal_2}. Note that by the definition of $h_m$ (see \eqref{def of hm}), $(h_m)^p \leq h_m$. By the definition of $\tau_m(S)$, Jensen's inequality, Young's inequality, and \eqref{embedding for vm}, we have
\begin{equation}\label{eq:bound of vm continued}
  \begin{aligned}
    \bE&\int_0^{\tau_m(S)}\int_{\bR^d}\left|  u^{1+\gamma}_+(t,\cdot) h_m(u_m(t,\cdot)) \right|^p dxdt \\
& \leq \bE\int_0^{\tau_m(S)}\int_{\bR^d}  u^{1+\beta}_+(t,\cdot) h_m(u_m(t,\cdot)) dxdt\sup_{t\leq\tau_m(S)}\sup_{x\in\bR^d}|u_m(t,x)|^{p(1+\gamma) - (1+\beta)} \\
& \leq N(S)\left( \bE\sup_{t\leq\tau_m(S)}\sup_{x\in\bR^d}|v_m(t,x)|^{p} \right)^{1/p_0} \\
& \leq N(\beta,\gamma,\ep,p,S) +  \ep\bE\sup_{t\leq\tau_m(S)}\sup_{x\in\bR^d}|v_m(t,x)|^{p}  \\
& \leq N(\beta,\gamma,\ep,p,S) + \ep N_0 \| v_m \|_{\cH_p^\kappa(\tau_m(S))}^p,
  \end{aligned}
\end{equation} 
where $p_0$ is the constant introduced in \eqref{def of p0}. Then, by taking $\ep>0$ satisfying $\ep N_0 < 1/2$ and applying \eqref{eq:bound of vm continued} into \eqref{eq:bound of vm}, we get 
\begin{equation*}
\| v_m \|_{\cH_p^\kappa(\tau_m(S))}^p 
\leq N_S + \frac{1}{2} \| v_m \|_{\cH_p^\kappa(\tau_m(S))}^p
\end{equation*} 
By subtracting $\frac{1}{2}\| v_m \|_{\cH_p^{1/2-\kappa}(\tau_m(S))}^p$ on both sides, we have $\| v_m \|_{\cH_p^\kappa(\tau_m(S))}^p \leq N_S$. Thus, by employing \eqref{embedding for vm} again with Jensen's inequality, we have the second inequality of \eqref{bound of vm - 2}. The first inequlaity follows from Lemma \ref{lem:upper bound of u_m by v_m}. This completes the proof.

\end{proof}

Now we show that the local solution does not explode in finite time.

\begin{lemma}\label{lem:uniform bound of local solutions in probability}
Suppose all the conditions of Theorem \ref{thm:main} holds. Let $u_m$ be the solution introduced in Lemma \ref{cut_off_lemma}. Then, for any $T<\infty$, we have
\begin{equation}
\label{stopping_time_blow_up}
\lim_{R\to\infty}\sup_{m\in\bN}\bP\left( \left\{ \omega\in\Omega:\sup_{t\leq T,x\in\bR} |u_m(t,x)| > R \right\} \right) = 0.
\end{equation}
\end{lemma}
\begin{proof}
Let $T<\infty$. Suppose $v_m$ is the solution introduce in Lemma \ref{lem:well-posedness of v_m}. 
{By Lemma \ref{Lq bounds} and Chevyshev's inequality, we have
}
\begin{equation*}
\begin{aligned}
\bP\left( \tau_m(S) < T \right) 
&\leq \bP\left( \int_0^T\int_{\bR}|u_m(s,x)|^{1+\beta}h_m(u_m(s,x))dxds \geq S \right) \\
&\leq  S^{-1}\bE\int_0^T\int_{\bR}|u_m(s,x)|^{1+\beta}h_m(u_m(s,x))dxds \\
&\leq NS^{-1},
\end{aligned}
\end{equation*}
where $N$ is independent of $m$.
Thus, by Chebyshev's inequality {and Lemma \ref{lem:noise related solution control}}, we have
\begin{equation*}
\begin{aligned}
\bP\left( \sup_{t\leq T,x\in\bR^d}|u_m(t,x)| > R \right) 
&\leq \bP\left( \sup_{t\leq\tau_m(S)\wedge T,x\in\bR^d} |u_m(t,x)| > R \right) + \bP\left( \tau_m(S) < T \right) \\
& \leq R^{-p}\bE\sup_{t\leq \tau_m(S)\wedge T,x\in\bR^d}|u_m(t,x)|^p + NS^{-1}, \\
&\leq N_SR^{-p} + NS^{-1}, \\
\end{aligned}
\end{equation*}
where $N_S$ is independent of $m$ and $N$ is independent of $m$ and $S$. Therefore, by letting $R\to\infty$ and $S\to\infty$ in order, the lemma is proved.

\end{proof}

\vspace{0.3cm}

\subsection{Proof of Theorem \ref{thm:main}}\label{subsec:proof of main thm}
We are ready to prove Theorem \ref{thm:main}.
\begin{proof}[\textbf{Proof of Theorem \ref{thm:main}}]

{\it (Step 1). (Uniqueness). }
Suppose $u,\bar u\in \cH_{p,loc}^{\kappa}$ are nonnegative solutions of \eqref{main_equation}. By Definition \ref{definition_of_sol_space} \eqref{def_of_local_sol_space}, there are bounded stopping times $\tau_n$, $n = 1,2,\cdots$ 
such that
\begin{equation*}
\tau_n\uparrow\infty\quad\mbox{and}\quad u, \bar u \in \cH_{p}^{\kappa}(\tau_n).
\end{equation*}
Fix $n\in\bN$. Since $p > \frac{3}{\kappa}$, by Corollary \ref{embedding_corollary}, we have $u,\bar{u} \in C([0,\tau_n];C(\bR^d))$ (a.s.) and
\begin{equation}
\label{embedding in the proof of uniqueness_infinite_noise_1}
\bE\sup_{t\leq\tau_n}\sup_{x\in\bR^d}|u(t,x)|^p + \bE\sup_{t\leq\tau_n}\sup_{x\in\bR^d}|\bar u(t,x)|^p < \infty.
\end{equation}
For $m \in \bN$, define
\begin{equation*}
\begin{gathered}
\tau_{m,n}^1:=\inf\left\{t\geq0:\sup_{x\in\bR^d}|u(t,x)|> m\right\}\wedge\tau_n, \quad \tau_{m,n}^2:=\inf\left\{t\geq0:\sup_{x\in\bR^d}|\bar u(t,x)|> m\right\}\wedge\tau_n,
\end{gathered}
\end{equation*}
and 
\begin{equation} 
\label{stopping_time_cutting}
\tau_{m,n}:=\tau_{m,n}^1\wedge\tau_{m,n}^2.
\end{equation}
Due to \eqref{embedding in the proof of uniqueness_infinite_noise_1}, $\tau_{m,n}^1$ and $\tau_{m,n}^2$ are well-defined stopping times, and thus $\tau_{m,n}$ is a stopping time.
Observe that $u,\bar u\in\cH_p^{\kappa}(\tau_{m,n})$ and  $\tau_{m,n}\uparrow \tau_n$ as $m\to\infty$ almost surely. Fix $m\in\bN$. Notice that both $u,\bar u\in\cH_p^{\kappa}(\tau_{m,n})$ are solutions to 
\begin{equation*}
dv = \left(  Lv - \bar{b}  v_+^{1+\beta}h_m(v)  \right)\,dt + \xi v_+^{1+\gamma}{ h_m(v)}(f\ast e_k)  dw^k_t, \quad 0<t\leq\tau_{m,n};\quad v(0,\cdot)=u_0,
\end{equation*}
{where $Lv = a^{ij}v_{x^ix^j} + b^i v_{x^i} + cv$.}
By the uniqueness result in Lemma \ref{cut_off_lemma}, we conclude that $u=\bar u$ in $\cH_p^{\kappa}(\tau_{m,n})$ for each $m\in\bN$. The monotone convergence theorem yields $u=\bar u$ in $\cH_p^{\kappa}(\tau_n)$. Since $n\in\bN$ is arbitrary, this implies $u = \bar u$ in $\cH_{p,loc}^{\kappa}$.
\\

{\it (Step 2). (Existence).} The idea of the proof folllows from \cite[Section 8.4]{kry1999analytic} and \cite[Theorem 2.11]{han2019boundary}.
By Lemma \ref{cut_off_lemma}, there exists a nonnegative solution $u_m\in\cH_{p}^{\kappa}(T)$ satisfying \eqref{eq:cut_off_equation}. By Corollary \ref{embedding_corollary}, we have $u_m\in C([0,T];C(\bR^d))$ (a.s.) and
\begin{equation*}
\bE\sup_{t\leq T}\sup_{x\in\bR^d}|u_m(t,x)| < \infty
\end{equation*}
For $R\in\{ 1,2,\dots,m-1 \}$, define
\begin{equation} \label{stopping_time_taumm_1}
\tau_m^R 
:= \inf\left\{ t\geq0 : \sup_{x\in\bR^d} |u_m(t,x)|\geq R \right\}
\end{equation}
Then, $\tau_m^R$ is a well-defined stopping time. Additionally, observe that 
\begin{equation}
\label{monotonicity of stopping times}
\tau_R^R = \tau_m^R \leq \tau_m^m.
\end{equation}
Indeed, since $\sup\limits_{x\in\bR^d}|u_m(t,x)|\leq R$ for $t\leq \tau_m^R$,  we have $u_m\wedge m=u_m\wedge m\wedge R  = u_m \wedge R$ for $t\leq\tau_m^R$.  Thus, both $u_m$ and $u_R$ satisfy
\begin{equation*}
du = \left(Lu - \bar{b}u^{1+\beta}h_R(u) \right) dt + \xi u^{1+\gamma}h_R(u)   dw_t^k, \quad0<t\leq\tau_m^R
\end{equation*}
with initial data $u(0,\cdot)=u_0$. On the other hand, since $R < m$, $u_R\wedge R = u_R\wedge R\wedge m = u_R\wedge m$ for $t\leq \tau_R^R$. Then $u_m$ and $u_R$ satisfy
\begin{equation*}
du = \left(Lu - \bar{b}u^{1+\beta}h_R(u) \right) dt + \xi u^{1+\gamma}h_R(u)   dw_t^k,   dw_t^k, \quad0<t\leq\tau_R^R
\end{equation*}
with initial data $u(0,\cdot)=u_0$. Thus, we have $u_m=u_R$ in $\cH_p^{\kappa}\left((\tau_m^R\vee\tau_R^R)\wedge T\right)$ by the uniqueness in Lemma \ref{cut_off_lemma}. 
Then, for $t\leq\tau_m^R$, 
\begin{equation*}
\sup_{s\leq t}\sup_{x\in\bR^d}|u_R(s,x)|=\sup_{s\leq t}\sup_{x\in\bR^d}|u_m(s,x)|\leq R,
\end{equation*}
and this implies  $\tau_m^R\leq\tau_R^R$. Similarly, we have $\tau_m^R\geq\tau_R^R$. Besides, since $m > R$, we have $\tau_m^m \geq \tau_m^R$. Therefore, we have \eqref{monotonicity of stopping times}. 

Next, by \eqref{stopping_time_blow_up},
\begin{equation*}
\begin{aligned}
\limsup_{m\to\infty}  \bP\left( \tau_m^m < T \right) 
&\leq \limsup_{m\to\infty}\sup_{n} \bP\left(\sup_{t\leq T,x\in\bR^d}|u_m(t,x)|\geq m\right)\to 0,
\end{aligned}
\end{equation*}
which implies $\tau_m^m\to \infty$ in probability. Since $\tau_m^m$ is increasing,  we conclude that $\tau_m^m\uparrow \infty$ (a.s.). Define $\tau_m := \tau_m^m\wedge m$ and
\begin{equation*}
u(t,x):=u_m(t,x)\quad \text{for}~ t\in[0,\tau_m].
\end{equation*}
Note that $u$ satisfies \eqref{main_equation} for all $t \leq \tau_m$, because $|u(t,\cdot)| = |u_m(t,\cdot)|\leq m$ for $t\leq\tau_m$. Since $u = u_m \in \cH_p^{\kappa}(\tau_m)$ for any $m$, we have $u\in\cH_{p,loc}^{\kappa}$. This concludes the proof.
\end{proof}

\vspace{0.3cm}

\subsection{Proof of Theorem \ref{uniqueness of solution in p}}\label{subsec:uniqueness of solution in p}
In this section, we prove Theorem \ref{uniqueness of solution in p}.
\begin{proof}[\textbf{Proof of Theorem \ref{uniqueness of solution in p}}]

By the assumption, $u\in \cH_{p,loc}^\kappa$ is a unique solution to \eqref{main_equation}. By Definition \ref{definition_of_sol_space} and \eqref{holder regularity of solution}, there exists a sequence of stopping times $\{ \tau_l\}_{l\in \bN}$ such that $\tau_l \to \infty$ as $l\to\infty$ and 
\begin{equation*}
\sup_{t\leq\tau_l}\sup|u(t,x)| \leq l \quad (a.s.).
\end{equation*}
Then, we have
\begin{equation*}
\begin{aligned}
\int_0^{\tau_l}\int_{\bR^d}|u(t,x)|^qdxdt &= \int_0^{\tau_l}\int_{\bR^d}|u(t,x)|^p|u(t,x)|^{q-p}dxdt \\
& \leq \left(\sup_{t\leq\tau_l} \sup_{x\in\bR^d}|u(t,x)|\right)^{q-p}\int_0^{\tau_l}\|u(t,\cdot)\|^p_{L_p}dt \\
& \leq  l^{q-p}\int_0^{\tau_l}\|u(t,\cdot)\|^p_{H_p^{\kappa}}dt<\infty \quad \mbox{(a.s.)}.
\end{aligned}
\end{equation*}
Therefore, we can define a bounded stopping time
\begin{equation*}
\tau_{l,l_0}:=\tau_l\wedge\inf\left\{t>0:\int_0^{t}\|u(t,\cdot)\|^q_{L_q}dt >l_0\right\}
\end{equation*}
such that $\tau_{l,l_0}\uparrow \tau_l$ as $ l_0\to\infty$ and $u \in\bL_{q}(\tau_{l,l_0})$ for each $l_0\in\bN$. Thus, it suffices to show that $u\in \cH_q^\kappa(\tau_{l,l_0})$ to show $u\in \cH_{q,loc}^\kappa$.

Observe that a similar computation as in \eqref{cut_off_lower_order_cal_2} with $v = 0$ implies
\begin{equation}
\label{regularity_up}
\begin{aligned}
\left\| \xi u^{1+\gamma}(f\ast \bm{e}) \right\|_{\bH^{\kappa-1}_q(\tau_{l,l_0},\ell_2)}^q
\leq N\int_0^{\tau_{l,l_0}}\int_{\bR^d}|u(t,x)|^{q(1+\gamma)}dxdt 
\leq Nl^{q\gamma}\int_0^{\tau_{l,l_0}}\int_{\bR^d}|u(t,x)|^{q}dxdt. 
\end{aligned}
\end{equation}
Moreover, notice that
\begin{equation}
\label{regularity_up_1}
\left\| \bar b u^{1+\beta} \right\|_{\bH^{\kappa-2}_q(\tau_{l,l_0},\ell_2)}^q
\leq N\int_0^{\tau_{l,l_0}}\int_{\bR^d}|u(t,x)|^{q(1+\beta)}dxdt 
\leq Nl^{q\beta}\int_0^{\tau_{l,l_0}}\int_{\bR^d}|u(t,x)|^{q}dxdt. 
\end{equation}
Notice that \eqref{regularity_up} and \eqref{regularity_up_1} imply $\xi u^{1+\gamma} (f\ast \bm{e})\in \bH_q^{\kappa-1}(\tau_{l,l_0},\ell_2)\subset \bH_q^{-1}(\tau_{l,l_0},\ell_2)$ and $\bar b u^{1+\beta} \in \bH_q^{\kappa-2}(\tau_{l,l_0})\subset \bH_q^{-2}(\tau_{l,l_0})$, respectively. Thus, $u\in\cH_q^0(\tau_{l,l_0})$ since $u_0\in U_q^\kappa \subset U_q^0$, $a^{ij}u_{x^ix^j}+b^iu_{x^i}+cu\in\bH_q^{\kappa-2}(\tau_{l,l_0})\subset\bH_q^{-2}(\tau_{l,l_0})$. Therefore, Theorem \ref{theorem_nonlinear_case} yields that there exists a unique solution $\hat u\in\cH^{\kappa}_q(\tau_{l,l_0})$ to the SPDE
\begin{equation*}
d\hat u=\left(a^{ij}\hat u_{x^ix^j} + b^i \hat u_{x^i} + c \hat u - \bar b u^{1+\beta}\right)\,dt+\sum_{k=1}^{\infty}\xi u^{1+\gamma} (f\ast e_k)dw^k_t, \quad 0<t\leq\tau_{l,l_0}; \quad \hat u(0,\cdot)=u_0.
\end{equation*}
It should be remarked that the nonlinear terms of the above equation are $-\bar b u^{1+\beta}$ and $\xi u^{1+\lambda} (f\ast e_k)$, not $\bar b \hat u^{1+\beta}$ and $\xi\hat u^{1+\lambda} (f\ast e_k)$.  Since $\varphi := u-\hat u\in \cH_q^0(\tau_{l,l_0})$ satisfies the deterministic PDE
\begin{equation*}
 d\varphi=\left( a^{ij}\varphi_{x^ix^j} + b^i \varphi_{x^i} + c \varphi \right)\,dt, \quad0<t\leq\tau_{l,l_0}; \quad \varphi(0,\cdot)=0,
\end{equation*}
the deterministic version of Theorem \ref{theorem_nonlinear_case} (see \cite{ladyvzenskaja1988linear}) implies  $u=\hat u$ almost all $(\omega,t,x)$. Therefore, $u\in\cH^{\kappa}_q(\tau_{l,l_0})$. This proves the theorem.

\end{proof}

\vspace{2mm}

\section*{Acknowledgement}
J. Yi was supported by Samsung Science and Technology Foundation under Project Number SSTF-BA1401-51. B.-S. Han was supported by the National Research Foundation of Korea (NRF) grant
funded by the Korea government (MSIT) (No. NRF-2021R1C1C2007792).

\bibliographystyle{alpha}		

\newcommand{\etalchar}[1]{$^{#1}$}

\end{document}